\documentclass[a4paper,12pt,reqno]{amsart}

\usepackage{graphicx}
\usepackage{amsmath}
\usepackage{amscd}
\usepackage{amssymb}
\usepackage{mathrsfs}
\usepackage[left=2.5cm,right=2.5cm,bottom=3cm,top=3cm]{geometry}
\newtheorem{thm}{Theorem}
\newtheorem{setup}{Setup}
\newtheorem{conj}{Conjecture}
\newtheorem{lem}{Lemma}

\newtheorem{prop}{Proposition}
\newtheorem{rem}{Remark}
\numberwithin{equation}{section}
\begin{document}

\title{Collapsing limits of the K\"ahler-Ricci flow and the continuity method}

\author{Yashan Zhang}
\address{Beijing International Center for Mathematical Research, Peking University, Beijing 100871, China}
\email{yashanzh@pku.edu.cn}
\thanks{The author is partially supported by the Project MYRG2015-00235-FST of the University of Macau.}

\begin{abstract}
We consider the K\"ahler-Ricci flow on certain Calabi-Yau fibration, which is a Calabi-Yau fibration with one dimensional base or a product of two Calabi-Yau fibrations with one dimensional bases. Assume the K\"ahler-Ricci flow on total space admits a uniform lower bound for Ricci curvature, then the flow converges in Gromov-Hausdorff topology to the metric completion of the regular part of generalized K\"ahler-Einstein current on the base, which is a compact length metric space homeomorphic to the base. The analogue results for the continuity method on such Calabi-Yau fibrations are also obtained. Moreover, we show the continuity method starting from a suitable K\"ahler metric on the total space of a Fano fibration with one dimensional base converges in Gromov-Hausdorff topology to a compact metric on the base. During the proof, we show the metric completion of the regular part of a generalized K\"ahler-Einstein current on a Riemann surface is compact.
\end{abstract}

\maketitle

\section{Introduction}

\subsection{Setups}
In this paper, we discuss collapsing limits of the K\"ahler-Ricci flow and the continuity method. Let's begin by some background and motivation for our results. \par The general setup in this paper is as follows.
\begin{setup}\label{setup0}
Let $f:X\to Z$ be a holomorphic map between two compact K\"ahler manifolds with image $Y:=f(X)\subset Z$ an irreducible normal subvariety of $Z$ and $0<k:=dim(Y)<n:=dim(X)$ and $f:X\to Y$ is of connected fibers. Assume there exists a family of K\"ahler metrics $\omega(t)$, $t\in[0,T)$ and $0<T\le\infty$, on $X$ with the property that, as $t\to T^-$,
\begin{equation}\label{semiample0}
[\omega(t)]\to f^*[\chi]
\end{equation}
in $H^{1,1}(X,\mathbb R)$ for some K\"ahler class $[\chi]$ on $Z$ with a K\"ahler metric representative $\chi$.
\end{setup}

Then our main goal is to study the limit of $\omega(t)$ as $t\to T^-$.
\par For later convenience, we also set

\begin{setup}\label{setup1}
Assume Setup \ref{setup0} and additionally $[\chi]$ is rational, i.e. $[\chi]=2\pi c_1(L)$ for some ample $\mathbb Q$-line bundle $L$ on $Z$.
\end{setup}

\begin{setup}\label{setup2}
Assume Setup \ref{setup0} and additionally $Y$ is smooth and the set of critical values of $f$ has simple normal crossing support.
\end{setup}

Setup \ref{setup1} is the most natural case, in view of semi-ample fibration theorem \cite{La}. We will discuss a special case of Setup \ref{setup2}, see Setup \ref{setup3} later (also see Remark \ref{ending}).
\par In this paper, $\omega(t)$ will satisfy the K\"ahler-Ricci flow or the continuity method.

\subsection{The K\"ahler-Ricci flow}\label{KRFsubsect} Let's first consider the K\"ahler-Ricci flow case. Assume $\omega(t)_{t\in[0,\infty)}$ is a long time solution to the K\"ahler-Ricci flow

\begin{equation}\label{KRF0}
\left\{
\begin{aligned}
\partial_t\omega(t)&=-Ric(\omega(t))-\omega(t)\\
\omega(0)&=\omega_0,
\end{aligned}
\right.
\end{equation}
where $\omega_0$ is an arbitrary K\"ahler metric on $X$. We know the K\"ahler class along the K\"ahler-Ricci flow \eqref{KRF0} satisfies
$$[\omega(t)]=e^{-t}[\omega_0]+2\pi(1-e^{-t})c_1(K_X)$$
and hence in this case the condition \eqref{semiample0} in Setup \ref{setup0} means
\begin{equation}\label{semiample1}
2\pi c_1(K_X)=f^*[\chi].
\end{equation}
Note that condition \eqref{semiample1} implies the generic fiber of $f:X\to Y$ is a smooth Calabi-Yau manifold and so we may call it a Calabi-Yau fibration.
Following Song and Tian \cite{ST06,ST12}, we now recall a construction of a class of positive current, i.e. the generalized K\"ahler-Einstein current on $Y$ in the class $[\chi]$ (here and in the followings, we use the same notation $\chi$ to denote the restriction of $\chi$ to $Y$). Fix a smooth positive volume form $\Omega$ on $X$ with $\sqrt{-1}\partial\bar\partial\log\Omega=f^*\chi$. Denote $S\subset Y$ be the singular set of $Y$ together with the critical values of $f$ and $X_{reg}:=f^{-1}(Y\setminus S)$. Then by \cite[Lemma 3.1]{ST06} we find a closed real $(1,1)$-form $\omega_{SRF}$ on $X_{reg}$ such that, restricting to a smooth fiber $X_y$ for $y\in Y\setminus S$, $\omega_{SRF}|_{X_y}$ is the unique Ricci-flat K\"ahler metric in the class $[\omega_0|_{X_y}]$. Define $F:=\frac{\Omega}{\binom{n}{k}\omega_{SRF}^{n-k}\wedge f^*\chi^k}$, which is constant along every smooth fiber and hence is also a positive smooth function on $Y\setminus S$. Moreover, by \cite[Lemma 3.3, Proposition 3.2]{ST12} $F$ satisfies
\begin{equation}\label{equation}
F=\frac{1}{V_0}\frac{f_*\Omega}{\binom{n}{k}\chi^k}
\end{equation}
on $Y\setminus S$, where $V_0:=\int_{X_y}(\omega_0|_{X_y})^{n-k}$, $y\in Y\setminus S$, is a positive constant and
\begin{equation}
0<\delta\le F\in L^{1+\epsilon}(Y,\chi^k)
\end{equation}
on $Y$, where $\delta$ and $\epsilon$ are two positive constants. Then Song and Tian \cite{ST06,ST12} consider the following complex Monge-Amp\`{e}re equation on $Y$:
\begin{equation}\label{equation0}
(\chi+\sqrt{-1}\partial\bar\partial\psi)^k=e^{\psi}F\chi^k.
\end{equation}
Thanks to \cite{EGZ,ST12} (building on \cite{Au,Y,Ko1}), \eqref{equation0} admits a unique solution $\psi\in PSH(Y,\chi)\cap L^\infty(Y)\cap C^\infty(Y\setminus S)$ (strictly speaking, \cite{ST12} proved the case that $[\chi]$ is rational on $Z$; however, noting \cite[Theorem 3.4, Lemma 3.5]{DP04}, their arguments also apply without assuming rationality of $[\chi]$). Set $\omega_Y:=\chi+\sqrt{-1}\partial\bar\partial\psi$, which is a positive $(1,1)$-current on $Y$ and a smooth K\"ahler metric on $Y\setminus S$. We call $\omega_Y$ the generalized K\"ahler-Einstein current on $Y$, since it satisfies a generalized K\"ahler-Einstein equation on $Y\setminus S$, see \cite{ST06,ST12}. We also remark that $\omega_Y$ does not depend on the constant factor $V_0$ (and$\binom{n}{k}$) on the right hand side of \eqref{equation0} and so we may assume $V_0=1$.\\

In the fundamental works of Song and Tian \cite{ST06,ST12}, it is shown in Setup \ref{setup1} that $\omega(t)_{t\in[0,\infty)}$, the solution to the K\"ahler-Ricci flow \eqref{KRF0}, converges to $f^*\omega_Y$ as currents on $X$. Precisely, \cite{ST06,ST12} proved the $L^1$-convergence for K\"ahler potentials, where the rationality of $[\chi]$ is needed in the argument (we mention that it is very natural to assume the rationality of $[\chi]$, in view of semi-ample fibration theorem), see \cite[Propositions 5.3, 5.4]{ST12}. This convergence has been improved to some stronger topology by several works, see \cite{FZ,TWY} and reference therein. In particular, assuming the Setup \ref{setup1}, it is proved in \cite[Theorem 1.2]{TWY} that for any $K'\subset\subset X_{reg}$, $\omega(t)\to f^*\omega_Y$ in $C^0(K',\omega_0)$-topology.

\par Here we remark an observation which will be used later.
\begin{prop}\label{observation2}\cite{ST06,ST12,TWY}
Assume $\omega(t)_{t\in[0,\infty)}$, the solution to the K\"ahler-Ricci flow \eqref{KRF0} on $X$, satisfies Setup \ref{setup2}. Then, as $t\to\infty$, $\omega(t)\to f^*\omega_Y$ as currents on $X$ and for any $K'\subset\subset X_{reg}$, $\omega(t)\to f^*\omega_Y$ in $C^0(K',\omega_0)$-topology.
\end{prop}
To see Proposition \ref{observation2}, we only need to prove a local uniform convergence away from a proper subvariety on $X$ for K\"ahler potentials, i.e. the analog of \cite[Proposition 5.4]{ST12}. This can be achieved by simply modifying the maximum principle arguments in \cite{ST12} since $S$ has simple normal crossing support. Once we have this fact, the conclusion in Proposition \ref{observation2} follows from the same arguments in \cite[Theorem 1.2]{TWY}. Proposition \ref{observation2} will be applied to Setup \ref{setup3} later.\\

\par There exists a natural conjecture proposed by Song and Tian \cite{ST06,ST12,ST17}.
\begin{conj}\label{conj0}\cite[Conjecture 6.3]{ST17}
Assume Setup \ref{setup0} with condition \eqref{semiample1}.
\begin{itemize}
\item[(1)] $(X_\infty,d_\infty)$, the metric completion of $(Y\setminus S,\omega_Y)$, is a compact metric space homeomorphic to $Y$.
\item[(2)] Assume Setup \ref{setup1} (or Setup \ref{setup2}) with condition \eqref{semiample1} and $\omega(t)_{t\in[0,\infty)}$ is the solution to the K\"ahler-Ricci flow \eqref{KRF0} on $X$. Then, as $t\to\infty$, $(X,\omega(t))\to(X_\infty,d_\infty)$ in Gromov-Hausdorff topology.
\end{itemize}
\end{conj}

We point out that the metric properties of $(Y\setminus S,\omega_Y)$ should be crucial in obtaining Gromov-Hausdorff convergence for the K\"ahler-Ricci flow. Our first main result deals with item (1) in Conjecture \ref{conj0} when $dim(Y)=1$ (i.e. $Y$ is a compact Riemann surface). Recall that, if $dim(Y)=1$, then $S$, the set of critical points of $f$, is a finite set of isolated points on $Y$. We denote $S=\{s_1,\ldots,s_L\}$. Our first main result can be stated as follows.

\begin{thm}\label{theorem1}
Assume Setup \ref{setup0} with condition \eqref{semiample1} and $dim(Y)=1$. Let $\omega_Y$ be the positive current obtained by \eqref{equation0}. Denote $S=\{s_1,\ldots,s_L\}$ be the critical values of $f$ on $Y$. For a fixed $s_l\in S$, we choose a local chart $(\Delta,s)$ near $s_l$ and identify $\Delta$ with $\{s\in\mathbb{C}||s|\le1\}$ and $s_l$ with $0\in\mathbb{C}$. Denote $\Delta^*=\Delta\setminus\{0\}$. Then, after possibly shrinking $\Delta$, there exist three positive constants $\beta_l\in\mathbb{Q}_{>0}$, $N_l\in\{1,2,\ldots,n\}$ and $C_l\ge1$ such that
\begin{equation}\label{est1.1}
C_l^{-1}|s|^{-2(1-\beta_l)}(-\log|s|)^{N_l-1}\le\frac{\omega_Y}{\sqrt{-1}ds\wedge d\bar s}(s)\le C_l|s|^{-2(1-\beta_l)}(-\log|s|)^{N_l-1}
\end{equation}
holds on $\Delta^*$. Consequently, $(X_\infty,d_\infty)$, the metric completion of $(Y\setminus S,\omega_{Y})$, is a compact length metric space and $X_\infty$ is homeomorphic to $Y$.
\end{thm}

\begin{rem}
We have to remark that if $dim(X)=2$, i.e. $X$ is a minimal elliptic surface, then the conclusion in Theorem \ref{theorem1} is first proved in Song and Tian \cite[Lemma 3.4]{ST06} and Hein \cite[Section 3.3]{He} by using Kodaira's classification on singular fibers of minimal elliptic surfaces (see e.g. \cite{Be}). In the general case $dim(X)\ge3$, since in general there is no classification on singular fibers, it seems Song-Tian-Hein's method dose't work. In the proof of Theorem \ref{theorem1}, our approach doesn't involve classification on singular fibers and hence is very different from Song and Tian \cite[Lemma 3.4]{ST06} and Hein \cite[Section 3.3]{He}. In particular, our proof also provides an alternative argument for the minimal elliptic surface case, see Section \ref{sectgke} for details.
\end{rem}

Having Theorem \ref{theorem1}, we can check Conjecture \ref{conj0} by assuming a lower bound for Ricci curvature. The following is our next main result.

\begin{thm}\label{result0}
Assume $\omega(t)_{t\in[0,\infty)}$, a solution to the K\"ahler-Ricci flow \eqref{KRF0} on $X$, satisfies Setup \ref{setup0} and $dim(Y)=1$. Assume Ricci curvature of $\omega(t)$ is uniformly bounded from below. Then, as $t\to\infty$, $(X,\omega(t))$ converges in Gromov-Hausdorff topology to $(X_\infty,d_\infty)$, the metric completion of $(Y\setminus S,\omega_Y)$.
\end{thm}
Note that the setting of Theorems \ref{theorem1} and \ref{result0} applies to any compact K\"ahler manifold with semi-ample canonical line bundle and Kodaira dimension one.\\
\par When $dim(X)=2$, i.e. $X$ is a minimal elliptic surface, Theorem \ref{result0} is proved in our previous work joint with Z.L. Zhang \cite{ZyZz1}, where we made crucial use of the result of Song and Tian \cite[Lemma 3.4]{ST06} and Hein \cite[Section 3.3]{He}, i.e. surface case in Theorem \ref{theorem1}. Now, as we have Theorem \ref{theorem1}, the argument in \cite{ZyZz1} can be carried out to obtain Theorem \ref{result0}. We will provide a sketch in Section \ref{proof.1}.\\

We can further extend Theorem \ref{result0} to more settings, which are products of Calabi-Yau fibrations with one dimensional bases. Precisely, we set
\begin{setup}\label{setup3}
Let $f:X\to\Sigma$ (resp. $f':X'\to\Sigma'$) be a holomorphic surjective map from an $n$-dimensional (resp. $m$-dimensional) compact connected K\"ahler manifold $X$ (resp. $X'$) to a compact connected Riemann surface $\Sigma$ (resp. $\Sigma'$) with connected fibers and $f^*\chi\in2\pi c_1(K_X)$ (resp. $f'^*\chi'\in2\pi c_1(K_{X'})$) for some K\"ahler metric $\chi$ (resp. $\chi'$) on $\Sigma$ (resp. $\Sigma'$). Set $S\subset\Sigma$ (resp. $S'\subset \Sigma'$) be the critical values of $f:X\to\Sigma$ (resp. $f':X'\to\Sigma'$), $\tilde X:=X\times X'$, $\tilde\Sigma:=\Sigma\times \Sigma'$, $\tilde\Sigma_{reg}:=(\Sigma\setminus S)\times(\Sigma'\setminus S')$ and $\tilde f:=(f,f'):\tilde X\to\tilde\Sigma$. We know there exists a K\"ahler metric $\tilde\chi:=\chi+\chi'$ on $\tilde\Sigma$ such that $\tilde f^*\tilde\chi\in2\pi c_1(K_{\tilde X})$. Given an arbitrary K\"ahler metric $\omega_0$ on $\tilde X$, the smooth solution $\omega(t)$ to the K\"ahler-Ricci flow \eqref{KRF0} on $\tilde X$ satisfies Setup \ref{setup2}. Moreover, we also have a unique generalized K\"ahler-Einstein current $\omega_{\tilde\Sigma}$ in K\"ahler class $\tilde\chi$ on $\tilde \Sigma$ and $\omega_{\tilde\Sigma}$ is a K\"ahler metric on $\tilde\Sigma_{reg}$.
\end{setup}

We have the following theorem.

\begin{thm}\label{result}
Given $\tilde f:\tilde X\to\tilde\Sigma$ as in above Setup \ref{setup3}. Let $\omega(t)_{t\in[0,\infty)}$ be the solution to the K\"ahler-Ricci flow \eqref{KRF0} on $\tilde X$.
\begin{itemize}
\item[(1)] $(\tilde X_\infty,\tilde d_\infty)$, the metric completion of $(\tilde\Sigma_{reg},\omega_{\tilde\Sigma})$, is a compact length metric space homeomorphic to $\tilde\Sigma$.
\item[(2)] Assume Ricci curvature of $\omega(t)$ is uniformly bounded from below. Then, as $t\to\infty$, $(\tilde X,\omega(t))\to (\tilde X_\infty,\tilde d_\infty)$ in Gromov-Hausdorff topology.
\end{itemize}
\end{thm}

\begin{rem}
On the one hand, the proofs of both Theorems \ref{result} and \ref{result0} make use of Theorem \ref{theorem1} crucially. On the other hand, the proof of Theorem \ref{result} is more involved than Theorem \ref{result0} in the sense that, to prove Theorem \ref{result}, we have to to apply some results in Cheeger-Colding's theory on Ricci limit space \cite{CC1,CC2}, see Section \ref{proof} for details.
\end{rem}

\subsection{The continuity method: collapsing limits at infinite time}
Next, we move to the continuity method case. The continuity method we will consider in this paper is the one proposed by La Nave and Tian \cite{LT} (also see \cite{Ru}), which is proposed to carry out the Analytic Minimal Model Program, see \cite{LT} for more details. Given a compact K\"ahler manifold $X$ and an arbitrary K\"ahler metric $\omega_0$ on $X$, assume $\omega(t)_{t\in[0,\infty)}$ be the smooth long time solution to the following equation of La Nave and Tian \cite{LT}:

\begin{equation}\label{LT0}
\left\{
\begin{aligned}
(1+t)\omega(t)&=\omega_0-tRic(\omega(t))\\
\omega(0)&=\omega_0,
\end{aligned}
\right.
\end{equation}
Easily, we find the K\"ahler class $[\omega(t)]$ along the continuity method \eqref{LT0} satisfies
\begin{equation}
[\omega(t)]=\frac{1}{1+t}\omega_0+\frac{t}{1+t}2\pi c_1(K_X).
\end{equation}
Therefore, when we assume $\omega(t)_{t\in[0,\infty)}$ satisfies Setup \ref{setup0}, there also holds \eqref{semiample1} and so we can follow discussions in subsection \ref{KRFsubsect} (e.g. we have the same generalized K\"ahler-Einstein current $\omega_Y$ solving \eqref{equation0} on $Y$), but replace the K\"ahler-Ricci flow \eqref{KRF0} by the continuity method \eqref{LT0}.
\par For the continuity method \eqref{LT0}, our first observation is the following
\begin{prop}\label{observation3}
Assume $\omega(t)_{t\in[0,\infty)}$, the solution of the continuity method \eqref{LT0}, satisfies Setup \ref{setup0}. Then, as $t\in\infty$, $\omega(t)\to f^*\omega_Y$ as currents on $X$.
\end{prop}

In Proposition \ref{observation3}, if $dim(X)=2$, then that conclusion was obtained in \cite{ZyZz1} (whose argument is modified from \cite{To10}). Moreover, we can modify arguments in \cite{ZyZz1} to check Proposition \ref{observation3} in general case. Precisely, we shall obtain an $L^1$-convergence of K\"ahler potentials and then Proposition \ref{observation3} follows. Moreover, as we will have $C^2$-estimate away from singular fibers for K\"ahler potential, we further improve the $L^1$-convergence to $C^{1,\alpha}$-convergence away from singular fibers. \\

\par We also have $C^0$-convergence of metric along the continuity method \eqref{LT0}.

\begin{prop}\label{observation4}
Assume $\omega(t)_{t\in[0,\infty)}$, the solution of the continuity method \eqref{LT0}, satisfies Setup \ref{setup0}. Then for any $K'\subset\subset X_{reg}$, $\omega(t)\to f^*\omega_Y$ in $C^0(K',\omega_0)$-topology as $t\to\infty$.
\end{prop}

Proposition \ref{observation4} can be checked by the same arguments in \cite[Theorem 1.1]{ZyZz2}, which is technically motivated by arguments in \cite{TWY}.\\

\par We also have the following natural conjecture, proposed by La Nave and Tian \cite{LT}.

\begin{conj}\label{conj3}\cite[Conjecture 4.7]{LT}
Assume $\omega(t)_{t\in[0,\infty)}$, the solution of the continuity method \eqref{LT0}, satisfies Setup \ref{setup0}. Then, as $t\to\infty$, $(X,\omega(t))\to(X_\infty,d_\infty)$ in Gromov-Hausdorff topology. Here $(X_\infty,d_\infty)$ is the same metric space as in Conjecture \ref{conj0}.
\end{conj}

When $dim(X)=2$, i.e. $X$ is a minimal elliptic surfacehe, Conjecture \ref{conj3} was confirmed in \cite{ZyZz1}, where we made crucial use of the result of Song and Tian \cite[Lemma 3.4]{ST06} and Hein \cite[Section 3.3]{He}, i.e. surface case in Theorem \ref{theorem1}. Now, as we have Theorem \ref{theorem1}, we can prove the following main result.

\begin{thm}\label{result0.1}
Assume $\omega(t)_{t\in[0,\infty)}$, a solution to the continuity method \eqref{LT0} on $X$, satisfies Setup \ref{setup0} and $dim(Y)=1$. Then, as $t\to\infty$, $(X,\omega(t))$ converges in Gromov-Hausdorff topology to $(X_\infty,d_\infty)$, the metric completion of $(Y\setminus S,\omega_Y)$.
\end{thm}
The proof of Theorems \ref{result0.1} and \ref{result0} are the same. Note that for the continuity method \eqref{LT0}, the Ricci curvature is uniformly bounded from below automatically.
\par We also have the analog of Theorem \ref{result} for the continuity method \eqref{LT0} as follows.

\begin{thm}\label{result.1}
Given $\tilde f:\tilde X\to\tilde\Sigma$ as in Setup \ref{setup3}. Let $\omega(t)_{t\in[0,\infty)}$ be the unique smooth solution to the continuity method  \eqref{LT0} on $\tilde X$. Then, as $t\to\infty$, $(\tilde X,\omega(t))\to (\tilde X_\infty,\tilde d_\infty)$ in Gromov-Hausdorff topology.
\end{thm}

The proof of Theorem \ref{result.1} is almost identical as Theorem \ref{result}, see Section \ref{proof} for details.

\subsection{The continuity method: collapsing limits at finite time}
We now discuss a different case where the continuity method will collapse at finite time.
Let $f:X\to Y\subset Z$ be the same as in Setup \ref{setup0} and $\omega_0$ be a K\"ahler metric on $X$ satisfying
\begin{equation}\label{semiample2}
[\omega_0]+2\pi c_1(K_X)=f^*[\chi]
\end{equation}
for some K\"ahler metric $\chi$ on $Z$. Note that condition \eqref{semiample2} implies the generic fiber of $f:X\to Y$ is a smooth Fano manifold and so we may call it a Fano fibration. Consider the following continuity method on $X$ proposed by La Nave and Tian \cite{LT}:
\begin{equation}\label{LT0.1}
\left\{
\begin{aligned}
\omega(t)&=\omega_0-tRic(\omega(t))\\
\omega(0)&=\omega_0,
\end{aligned}
\right.
\end{equation}
Easily, the K\"ahler class $[\omega(t)]$ along the continuity method \eqref{LT0.1} satisfies
\begin{equation}
[\omega(t)]=[\omega_0]+2\pi t c_1(K_X).
\end{equation}
As we have condition \eqref{semiample2}, applying \cite[Theorem 1.1]{LT} gives a smooth solution $\omega(t)$ to the continuity method \eqref{LT0.1} on time interval $[0,1)$. Then $\omega(t)_{t\in[0,1)}$ solving \eqref{LT0.1} satisfies Setup \ref{setup0} with the property
$$[\omega(t)]\to f^*[\chi]$$
as $t\to1$.
\par The goal is to study the limit of $\omega(t)$ as $t\to1$. In general, according to \cite[Conjecture 4.1]{LT}, $(X,\omega(t))$ should converge to some compact metric on $Y$. Some progresses have been made in \cite{ZyZz2}. In particular, a limiting singular K\"ahler metric was constructed in \cite[Section 2]{ZyZz2}. Let's first recall the construction in \cite[Section 2]{ZyZz2}. By the condition \eqref{semiample2} and Yau's theorem \cite{Y}, we fix a smooth positive volume form $\Omega$ on $X$ with
\begin{equation}
\sqrt{-1}\partial\bar\partial\log\Omega=f^*\chi-\omega_0,
\end{equation}
and a closed real $(1,1)$-form $\overline\omega_0$ on $X_{reg}$ such that, restricting to every smooth fiber $X_y$ for $y\in Y\setminus S$, $\overline\omega_0|_{X_y}$ is the unique K\"ahler metric in class $[\omega_0|_{X_y}]$ satisfying $Ric(\overline\omega_0|_{X_y})=\omega_0|_{X_y}$. Then define a function $G$ on $X_{reg}$ by
$$G=\frac{\Omega}{\binom{n}{k}\overline\omega_0^{n-k}\wedge f^*\chi^k},$$
which is constant along every smooth fiber $X_y$ and hence is also a smooth positive function on $Y\setminus S$. Moreover, $G$ satisfies
\begin{equation}
G=\frac{1}{V_0}\frac{f_*\Omega}{\binom{n}{k}\chi^k}
\end{equation}
on $Y\setminus S$, where $V_0:=\int_{X_y}(\omega_0|_{X_y})^{n-k}$, $y\in Y\setminus S$, is a positive constant and
\begin{equation}
0<\delta\le G\in L^{1+\epsilon}(Y,\chi^k)
\end{equation}
on $Y$, where $\delta$ and $\epsilon$ are two positive constants. Now consider the following complex Monge-Amp\`{e}re equation on $Y$:
\begin{equation}\label{equation0.1}
(\chi+\sqrt{-1}\partial\bar\partial\psi)^k=e^{\psi}G\chi^k.
\end{equation}
Thanks to \cite{EGZ,ST12}, \eqref{equation0.1} admits a unique solution $\psi\in PSH(Y,\chi)\cap L^\infty(Y)\cap C^\infty(Y\setminus S)$. Set $\omega_Y:=\chi+\sqrt{-1}\partial\bar\partial\psi$, which is a positive $(1,1)$-current on $Y$ and a smooth K\"ahler metric on $Y\setminus S$. Then the following result is essentially contained in \cite{ZyZz2}.

\begin{prop}\label{observation4.1}\cite{ZyZz2}
Assume Setup \ref{setup0} and $\omega(t)_{t\in[0,1)}$ is the solution to the continuity method \eqref{LT0.1} with property \eqref{semiample2}. Then, as $t\to1$, $\omega(t)\to f^*\omega_Y$  as currents on $X$.
\end{prop}
In fact, \cite{ZyZz2} proved the $L^1$-convergence of K\"ahler potential, which takes place in $C^{1,\alpha}$ for any $\alpha\in(0,1)$ away from singular fibers. We refer reader to \cite{ZyZz2} for more details.
\par Similarly, we also have an analog of Proposition \ref{observation4} as follows.

\begin{prop}\label{observation5}\cite{ZyZz2}
Assume $\omega(t)_{t\in[0,1)}$, the solution of the continuity method \eqref{LT0.1}, satisfies Setup \ref{setup0}. Then for any $K'\subset\subset X_{reg}$, $\omega(t)\to f^*\omega_Y$ in $C^0(K',\omega_0)$-topology as $t\to1$.
\end{prop}

If $dim(Y)=1$, we also have the asymptotics of $\omega_Y$ on $Y$ near singular points.

\begin{thm}\label{theorem1.1}
Assume Setup \ref{setup2} with condition \eqref{semiample2} and $dim(Y)=1$. Let $\omega_Y$ be the positive current obtained by \eqref{equation0.1}. Denote $S=\{s_1,\ldots,s_L\}$ be the critical values of $f$ on $Y$. For a fixed $s_l\in S$, we choose a local chart $(\Delta,s)$ near $s_l$ and identify $\Delta$ with $\{s\in\mathbb{C}||s|\le1\}$ and $s_l$ with $0\in\mathbb{C}$. Denote $\Delta^*=\Delta\setminus\{0\}$. Then, after possibly shrinking $\Delta$, there exist three positive constants $\beta_l\in\mathbb{Q}_{>0}$, $N_l\in\{1,2,\ldots,n\}$ and $C_l\ge1$ such that
\begin{equation}\label{est1.1}
C_l^{-1}|s|^{-2(1-\beta_l)}(-\log|s|)^{N_l-1}\le\frac{\omega_Y}{\sqrt{-1}ds\wedge d\bar s}(s)\le C_l|s|^{-2(1-\beta_l)}(-\log|s|)^{N_l-1}
\end{equation}
holds on $\Delta^*$. Consequently, $(X_1,d_1)$, the metric completion of $(Y\setminus S,\omega_{Y})$, is a compact length metric space and $X_1$ is homeomorphic to $Y$.
\end{thm}

Finally, using Theorem \ref{theorem1.1}, we can determine Gromov-Hausdorff limit of the continuity method \eqref{LT0.1} on Fano fibration with one dimensional base.

\begin{thm}\label{result0.2}
Assume $\omega(t)_{t\in[0,1)}$, a solution to the continuity method \eqref{LT0.1} on $X$, satisfies Setup \ref{setup0} and $dim(Y)=1$. Then, as $t\to1$, $(X,\omega(t))$ converges in Gromov-Hausdorff topology to $(X_1,d_1)$, the metric completion of $(Y\setminus S,\omega_Y)$. Here $\omega_Y$ is solved from \eqref{equation0.1}.
\end{thm}
The proof of Theorem \ref{result0.2} is the same as Theorems \ref{result0} and \ref{result0.1}, see Section \ref{proof.1} for details.

\subsection{Organization of this paper} The remaining part of this paper is organized as follows. In Section \ref{estsect}, we collect some properties of the K\"ahler-Ricci flow and the continuity method. In Section \ref{sectgke}, we prove a general result Theorem \ref{prop3.1}, which implies Theorems \ref{theorem1} and \ref{theorem1.1}. In Section \ref{proof.1}, we prove Theorems \ref{result0}, \ref{result0.1} and \ref{result0.2}. In Section \ref{proof}, we prove Theorems \ref{result} and \ref{result.1}.

\section{Properties of the K\"ahler-Ricci flow and the continuity method}\label{estsect}
We collect some necessary properties of the K\"ahler-Ricci flow and the continuity method.

\subsection{The K\"ahler-Ricci flow}\label{krfsub}Assume we are given a solution $\omega(t)_{t\in[0,\infty)}$ to the K\"ahler-Ricci flow \eqref{KRF0} on $X$ satisfying Setup \ref{setup0}. We first reduce the K\"ahler-Ricci flow to a parabolic complex Monge-Amp\`{e}re equation. Let $\Omega$ be the same smooth positive volume form on $X$ with $\sqrt{-1}\partial\bar\partial\log\Omega=f^{*}\chi$ as in subsection \ref{KRFsubsect} and set $\omega_t:=e^{-t}\omega_0+(1-e^{-t}) f^*\chi$ and $\omega(t)=\omega_t+\sqrt{-1}\partial\bar\partial\varphi(t)$. Then we reduce K\"ahler-Ricci flow \eqref{KRF0} to
\begin{equation}\label{conti.eq}
\left\{
\begin{aligned}
\dot\varphi(t)&=\log\frac{e^{(n-k)t}(\omega_t+\sqrt{-1}\partial\bar\partial\varphi(t))^{n}}{\Omega}-\varphi(t)\\
\varphi(0)&=0.
\end{aligned}
\right.
\end{equation}
Thanks to \cite{FZ,Gi,ST06,ST12,ST4,SW13,TWY}, we have the following
\begin{prop}\label{facts}
Assume $\omega(t)_{t\in[0,\infty)}$, a solution to the K\"ahler-Ricci flow \eqref{KRF0} on $X$, satisfies Setup \ref{setup0}. Let $\varphi(t)$ be the solution to \eqref{conti.eq} on $X$ and hence $\omega(t)=\omega_t+\partial\bar\partial\varphi(t)$ is the solution to \eqref{KRF0}. There exists a uniform constant $C\ge1$ such that
\begin{itemize}
\item[(1)] $|\varphi|+|\dot{\varphi}|\le C$, and hence $C^{-1}e^{-(n-k)t}\Omega\le\omega(t)^{n}\le Ce^{-(n-k)t}\Omega$, on $X\times[0,\infty)$;
\item[(2)] $\omega(t)\ge C^{-1}f^*\chi$ on $X\times[0,\infty)$;
\end{itemize}
Moreover, given any $K\subset\subset X_{reg}$, there exist two constants $C\ge1$ and $T>0$ such that for all $t\ge T$,
\begin{itemize}
\item[(3)] $C^{-1}\omega_t\le\omega(t)\le C\omega_t$ on $K$;
\item[(4)] Assume Setup \ref{setup1} or Setup \ref{setup2}, we have $\varphi(t)\to f^*\psi$ in $L^1(X,\Omega)\cap C^{1,\alpha}_{loc}(X_{reg})$-topology, for any $\alpha\in(0,1)$.
\end{itemize}
\end{prop}
\begin{proof}
Items (1) and (2) are contained in \cite{ST06,ST12,ST4}. Item (4) is proved in \cite{ST06} when $n-k=1$ and in \cite{FZ} when $n-k\ge2$. Item (4) is proved in \cite{ST06,ST12,FZ}.

\end{proof}

\par We also need the following
\begin{lem}\label{L1conv}
Assume $\omega(t)_{t\in[0,\infty)}$, a solution to the K\"ahler-Ricci flow \eqref{KRF0} on $X$, satisfies Setup \ref{setup1} or Setup \ref{setup2}. Then, as $t\to\infty$, $e^{\dot\varphi(t)+\varphi(t)}\to e^{f^*\psi}$ in $L^1(X,\Omega)$-topology.
\end{lem}
\begin{proof}
Firstly, recall that $|\dot\varphi|+|\varphi|+|f^*\psi|$ is uniformly bounded on $X\times[0,\infty)$. Moreover, according to \cite[Lemma 3.2(iv)]{TWY}, for any $K\subset\subset X_{reg}$, $\dot\varphi(t)+\varphi(t)\to f^*\psi$ in $L^{\infty}(K)$-topology (its proof needs $L^1$-convergence of K\"ahler potentials in above item (4)). Then we can easily use an elementary argument to conclude this lemma.
\end{proof}

\subsection{The continuity method}
As the K\"ahler-Ricci flow case, we first reduce the continuity method \eqref{LT0} to a complex Monge-Amp\`{e}re equation. For convenience, we consider a reparametrization of continuity method \eqref{LT0} as follows;
\begin{equation}\label{lt1}
\left\{
\begin{aligned}
e^t\omega(t)&=\omega_0-(e^t-1)Ric(\omega(t))\\
\omega(0)&=\omega_0.
\end{aligned}
\right.
\end{equation}

Assume $\omega(t)_{t\in[0,\infty)}$ satisfies Setup \ref{setup0}. We use the same notation $\Omega,\omega_t$ as in subsection \ref{krfsub} and set $\omega(t)=\omega_t+\sqrt{-1}\partial\bar\partial\varphi(t)$, then \eqref{lt1} can be reduced to the following complex Monge-Amp\`{e}re equation

\begin{equation}\label{LT2}
\left\{
\begin{aligned}
\varphi(t)&=(1-e^{-t})\log\frac{e^{(n-k)t}(\omega_t+\sqrt{-1}\partial\bar\partial\varphi(t))^n}{\Omega}\\
\varphi(0)&=0.
\end{aligned}
\right.
\end{equation}

The following is an analog of Proposition \ref{facts} for the continuity method \eqref{LT2}.
\begin{prop}\label{facts.1}
Assume $\omega(t)_{t\in[0,\infty)}$, a solution to the continuity method \eqref{lt1} on $X$, satisfies Setup \ref{setup0}. Let $\varphi(t)$ be the solution to \eqref{LT2} on $X$ and hence $\omega(t)=\omega_t+\partial\bar\partial\varphi(t)$ is the solution to \eqref{lt1}. There exists a uniform constant $C\ge1$ such that
\begin{itemize}
\item[(1)] $|\varphi|\le C$, and hence $C^{-1}e^{-(n-k)t}\Omega\le\omega(t)^{n}\le Ce^{-(n-k)t}\Omega$, on $X\times[1,\infty)$;
\item[(2)] $\omega(t)\ge C^{-1}f^*\chi$ on $X\times[1,\infty)$;
\end{itemize}
Moreover, given any $K\subset\subset X_{reg}$, there exist two constants $C\ge1$ and $T>0$ such that for all $t\ge T$,
\begin{itemize}
\item[(3)] $C^{-1}\omega_t\le\omega(t)\le C\omega_t$ on $K$;
\item[(4)] $\varphi(t)\to f^*\psi$ in $L^1(X,\Omega)\cap C^{1,\alpha}_{loc}(X_{reg})$-topology, for any $\alpha\in(0,1)$.
\end{itemize}
\end{prop}
\begin{proof}
See \cite{ZyZz1,ZyZz2} for a proof.
\end{proof}

We also have an analog of Lemma \ref{L1conv}.
\begin{lem}\label{L1conv.1}
Assume $\omega(t)_{t\in[0,\infty)}$, a solution to the continuity method \eqref{lt1} on $X$, satisfies Setup \ref{setup0}. Then, as $t\to\infty$, $e^{\varphi(t)}\to e^{f^*\psi}$ in $L^1(X,\Omega)$-topology.
\end{lem}

\section{Limiting singular metrics on Riemann surfaces}\label{sectgke}
\subsection{A general result}In this subsection, we shall first prove the following result.
\begin{thm}\label{prop3.1}
Let $f:X\to\Sigma$ be a holomorphic surjective map from an $n$-dimensional ($n\ge2$) compact K\"ahler manifold $X$ to a compact Riemann surface $\Sigma$ with connected fibers. Let $S=\{s_1,\ldots,s_L\}$ be the critical values of $f$ on $\Sigma$. For a fix $s_l\in disc(f)$, we choose a local chart $(\Delta,s)$ near $s_l$ and identify $\Delta$ with $\{s\in\mathbb{C}||s|\le1\}$ and $s_l$ with $0\in\mathbb{C}$. Denote $\Delta^*=\Delta\setminus\{0\}$. Then, after possibly shrinking $\Delta$, there exist two positive constants $\beta_l\in\mathbb{Q}_{>0}$ and $N_l\in\{1,2,\ldots,n\}$ such that, for any smooth potitive volume form $\Omega$ on $X$, there exists a constant $C_l\ge1$ such that
\begin{equation}\label{est1}
C_l^{-1}|s|^{-2(1-\beta_l)}(-\log|s|)^{N_l-1}\le\frac{f_*\Omega}{\sqrt{-1}ds\wedge d\bar s}(s)\le C_l|s|^{-2(1-\beta_l)}(-\log|s|)^{N_l-1}
\end{equation}
holds on $\Delta^*$.
\end{thm}
This result is an analogue of the degeneration of $L^2$-metric on Hodge bundle $f_*K_{X/\Sigma}$, see e.g. \cite{Sc,Pe,BJ,EFM}. However, in general $f_*\Omega$ is not the $L^2$-metric on $f_*K_{X/\Sigma}$ and we can not apply the results in \cite{Sc,Pe,BJ,EFM} directly to this setting.
\begin{proof}[Proof of Theorem \ref{prop3.1}]
Since our result is local on the base $\Sigma$, in the following we assume without loss of generality that $S=\{s_1\}$, i.e., there is only one critical value. Set $\Sigma_{reg}:=\Sigma\setminus S$.
\par By applying a Hironaka's resolution of singularity, we obtain a birational morphism
$$\pi:\hat{X}\to X$$
such that $\hat X$ is smooth and
$$\hat f:=f\circ\pi:\hat X\to \Sigma$$
is a holomorphic surjective map with connected fibers and the only singular fiber
\begin{equation}
\hat X_0=\sum a_jE_j\nonumber,
\end{equation}
where $a_i\in\mathbb{N}_{\ge1}$, $E_i$'s are smooth irreducible divisors in $\hat X$ and have simple normal crossings. Since $\pi$ is a biholomorphism over $X\setminus X_0$, we know
\begin{equation}\label{e3.1}
\int_{X_s}\Omega=\int_{\hat X_s}\pi^*\Omega
\end{equation}
on $\Sigma_{reg}$. Therefore, we are led to compute the right hand side of \eqref{e3.1}. \\

\par To this end, we naturally need the following ramification formula
$$K_{\hat X}=\pi^*K_X+\sum k_jE_j,$$
where $k_j\in\mathbb{N}_{\ge0}$ since $X$ is smooth. If we fix a defining section $s_j$ of $E_j$ and a smooth Hermitian metric $h_j$ on $\mathcal{O}_{\hat X}(E_j)$, then there exists a smooth nondegenerate volume form $\hat \Omega$ on $\hat X$ satisfying
$$\pi^*\Omega=\prod_{j}|s_j|_{h_j}^{2k_j}\hat\Omega.$$

For any fixed point $x_0\in \hat X_0$, we set $J(x_0)=\{j|x_0\in E_j\}$. Then we choose a local chart $(U_{0},z^1,\ldots,z^n)$ on $\hat X$ near $x_0$ such that for $j\in J(x_0)$, $E_j$ is given by $z^j=0$ and the map $f$ is given by
$$\hat f(z^1,\ldots,z^n)=\prod_{j\in J(x_0)}(z^j)^{a_j}=s.$$
After possibly shrinking $U_{0}$, we write
$$\pi^*\Omega=\sqrt{-1}^{n^2}h^0\prod_{j\in J(x_0)}|z^j|^{2k_j}dz^1\wedge\ldots\wedge dz^n\wedge d\bar z^1\wedge\ldots\wedge d\bar z^n$$
for some positive smooth function $h^0$ on $U_{0}$. Moreover, for any $j_0\in J(x_0)$, if we set
$$\xi_0:=\frac{1}{s}(-1)^{j_0}\frac{z^{j_0}}{a_{j_0}}dz_1\wedge\ldots\wedge dz^{j_0-1}\wedge dz^{j_0+1}\wedge\ldots\wedge dz^{n},$$
then we have
$$\pi^*\Omega=\sqrt{-1}^{n^2}h^0\prod_{j\in J(x_0)}|z^j|^{2k_j}\hat f^*ds\wedge\xi_0\wedge\overline{\hat f^*ds}\wedge\bar\xi_0$$
on $U_0$ and hence
\begin{equation}\label{ae.2.1}
\frac{\int_{\hat X_s\cap U_0}\pi^*\Omega}{\sqrt{-1}ds\wedge d\bar s}=\sqrt{-1}^{(n-1)^2}\int_{\hat X_s\cap U_0}h^0\prod_{j\in J(x_0)}|z^j|^{2k_j}\xi_0\wedge\bar\xi_0.
\end{equation}
The most convenient choice of $j_0$ for the above $\xi_0$ will be determined later.\\

\par To compute the right hand side of \eqref{ae.2.1}, we now apply a trick in \cite[Section 2]{EFM}. More precisely, we change the variables by $z^j=e^{\frac{u^j}{a_j}}e^{\sqrt{-1}\theta^j}$ for $j\in J(x_0)$ and $z^j=r_je^{\sqrt{-1}\theta^j}$ for $j\in\{1,\ldots,n\}\setminus J(x_0)$. We may assume $u^j\le0$ and $0\le r^j\le1$ on $U_0$. Moreover, in this new coordinate,
$$\hat X_s\cap U_0=\{\sum_{j\in J(x_0)}u^j=\log|s|\}.$$
Then we have

\begin{align}\label{cal1}
&\sqrt{-1}^{(n-1)^2}\int_{\hat X_s\cap U_0}h^0\prod_{j\in J(x_0)}|z^j|^{2k_j}\xi_0\wedge\bar\xi_0\nonumber\\
&=\frac{1}{|s|^2}\int_{\hat X_s\cap U_0}h^{0}\prod_{b\in J(x_0)}|z^b|^{2k_b}\frac{|z_{j_0}|^2}{|a_{j_0}|^2}|dz^1|^2\cdot\ldots\cdot |dz^{j_0-1}|^2\cdot |dz^{j_0+1}|^2\cdot\ldots\cdot |dz^n|^2\nonumber\\
&=\frac{C}{|s|^2}\int_{\{\sum_{j\in J(x_0)} u^{j}=\log|s|\}}h^{0}\prod_{b\in J(x_0)}e^{2\frac{k_b+1}{a_{b}}u^b}\prod_{j'\in J(x_0)\setminus\{j_0\}}du^{j'}\prod_{j''\in \{1,\ldots,n\}\setminus J(x_0)}dr^{j''}\nonumber.
\end{align}

Note that $h^0$ is uniformly bounded from zero. Therefore, if we choose a $j_0\in J(x_0)$ such that $\frac{k_{j_0}+1}{a_{j_0}}=\min\{\frac{k_{j}+1}{a_j},j\in J(x_0)\}$, then
\begin{align}
&\frac{\int_{\hat X_s\cap U_0}\pi^*\Omega}{\sqrt{-1}ds\wedge d\bar s}\nonumber\\
&=\sqrt{-1}^{(n-1)^2}\int_{\hat X_s\cap U_0}h^0\prod_{j\in J(x_0)}|z^j|^{2k_j}\xi_0\wedge\bar\xi_0\nonumber\\
&=\frac{C}{|s|^{2(1-\frac{k_{j_0}+1}{a_{j_0}})}}\int_{\{\sum_{j\in J(x_0)} u^j=\log|s|\}}h^{0}\prod_{b\in J(x_0)}e^{2(\frac{k_b+1}{a_b}-\frac{k_{j_0}+1}{a_{j_0}})u^b}\prod_{j'\in J(x_0)\setminus\{j_0\}}du^{j'}\prod_{j''\in \{1,\ldots,n\}\setminus J(x_0)}dr^{j''}\nonumber,
\end{align}
which clearly has asymptotic
$$|s|^{-2(1-\frac{k_{j_0}+1}{a_{j_0}})}(-\log|s|)^{N_0-1},$$
where $N_0:=|\{j\in J(x_0)|\frac{k_j+1}{a_j}=\frac{k_{j_0}+1}{j_0}\}|$.\\

Now we set $\beta=\min\{\frac{k_{j}+1}{a_j}\}$ and $N\ge1$ be the maximal number such that there are $N$ $E_j$'s with $\frac{k_j+1}{a_j}=\beta$ and non-empty intersection. Of course $\beta\in\mathbb{Q}_{>0}$ and $N\in\{1,2,\ldots,n\}$.
\par We choose a $x_0$ in such an intersection and its open neighborhood $U_0$ with above properties. Furthermore, we extend $U_0$ to be an open cover $\{U_\alpha\}$ of $\hat f^{-1}(\Delta)$ (after possibly shrinking $\Delta$) such that every $U_{\alpha}$ centers at some $x_\alpha\in \hat X_0$ and satisfies the above properties. Then, using a partition of unity $\{\phi_\alpha\}$ subordinate to $\{U_{\alpha}\}$, we have
\begin{align}
\frac{\int_{\hat X_s}\pi^*\Omega}{\sqrt{-1}ds\wedge d\bar s}&=\sum_{\alpha}\frac{\int_{\hat X_s}\phi_\alpha\pi^*\Omega}{\sqrt{-1}ds\wedge d\bar s}\nonumber\\
&=\sum_{\alpha}\frac{\int_{\hat X_s\cap U_\alpha}\phi_\alpha\pi^*\Omega}{\sqrt{-1}ds\wedge d\bar s}\nonumber.
\end{align}
Therefore, we conclude that, there exists a constant $C\ge1$ such that
$$C^{-1}|s|^{-2(1-\beta)}(-\log|s|)^{N-1}\le\frac{\int_{\hat X_s}\pi^*\Omega}{\sqrt{-1}ds\wedge d\bar s}\le C|s|^{-2(1-\beta)}(-\log|s|)^{N-1}$$
holds on $\Delta^*$, which, combining \eqref{e3.1}, completes the proof of Theorem \ref{prop3.1}.
\end{proof}

\begin{rem}\label{remk1.2}
The constant $\beta_l$ in Theorem \ref{prop3.1} is the \emph{log-canonical threshold} of $X_{s_l}$ along $X$, in the sense of \cite[Definition 9.3.12, Example 9.3.16]{La}. In the special case that $f:X\to\Sigma$ is a minimal elliptic surface, this fact was first pointed out by Ivan Cheltsov to Hans-Joachim Hein.
\end{rem}

\subsection{Two special cases of Theorem \ref{prop3.1}}
We now look at two special cases of Theorem \ref{prop3.1}.
\par \textbf{Special case (1): Calabi-Yau setting.} When the total space $X$ is an $n$-dimensional Calabi-Yau manifold and $\Omega=\sqrt{-1}^{n^2}\eta\wedge\bar\eta$ for some nowhere vanishing holomorphic $(n,0)$-form on $X$, a (not necessarily optimal) upper bound of $f_*\Omega$ similar to \eqref{est1} was proved by Gross-Tosatti-Zhang \cite[Section 2]{GTZ} by using Hodge theory (see also \cite{Be,BJ,Wa} for some related discussions).
\par \textbf{Special case (2): minimal elliptic surfaces.} When $f:X\to\Sigma$ in Theorem \ref{prop3.1} is a minimal elliptic surface, the estimate \eqref{est1} was proved by Song-Tian \cite[Lemma 3.4]{ST06} and Hein \cite[Section 3.3]{He}. Moreover, in \cite{He,ST06}, the constants $\beta$ and $N$ are obtained precisely according to the types of singular fibers in Kodaira's table \cite[Section V.7]{BHPV}. For example, in \cite{He}, these constants are determined by a detailed analysis on the asymptotic behaviors of semi-flat metrics near the singular fibers of $X$.
\par In the following, we discuss how to recover the result of Song-Tian \cite[Lemma 3.4]{ST06} and Hein \cite[Section 3.3]{He} by our above arguments, and hence provide an alternative proof for their result.
\par\emph{Case 1: $X_0$ has simple normal crossing support.}
\par These are singular types $\textrm{mI}_0, \textrm{mI}_b (b\ge3), \textrm{I}_b^* (b\ge0), \textrm{II}^*,\textrm{III}^*$ and $\textrm{IV}^*$. In these cases we don't need any resolution (or, let $\pi=Id:X\to X$) and hence $k_j=0$ for all $j$.
\par (1.1) For singular type $\textrm{mI}_0,\textrm{I}_0^*,\textrm{II}^*, \textrm{III}^*, \textrm{IV}^*$, there is exactly one component of $X_0$ with maximal multiplicity $m,2,6,4,3$, respectively, and therefore, $N=1$ and $\beta=\frac{1}{m},\frac{1}{2},\frac{1}{6},\frac{1}{4},\frac{1}{3}$, respectively.
\par (1.2) For singular type $\textrm{mI}_b (b\ge3)$, every component has the same multiplicity $m$ and there exist two components that has non-empty intersections. Therefore, $\beta=\frac{1}{m}$ and $N=2$.
\par (1.3) For singular type $\textrm{I}_b^* (b\ge1)$, there always exist two components that has the maximal multiplicity $2$ and non-empty intersections. Therefore, $\beta=\frac{1}{2}$ and $N=2$.\\

\par\emph{Case 2: $X_0$ doesn't have simple normal crossing support.}
\par These are singular types $\textrm{II},\textrm{III},\textrm{IV},\textrm{mI}_1$ and $\textrm{mI}_2$.
\par (2.1) For singular type $\textrm{II}$, i.e., $X_0$ is a cuspidal rational curve, we can choose a resolution $\pi:\hat X\to X$ such that
$$\hat X_0=E_1+2E_2+3E_3+6E_4$$
and
$$K_{\hat X}=\pi^*K_X+E_2+2E_3+4E_4.$$
Therefore, we have $\beta=\frac{5}{6}$ and $N=1$.
\par (2.2) For singular type $\textrm{III}$, we can choose a resolution $\pi:\hat X\to X$ such that
$$\hat X_0=E_1+E_2+2E_3+4E_4$$
and
$$K_{\hat X}=\pi^*K_X+E_3+2E_4.$$
Therefore, we have $\beta=\frac{3}{4}$ and $N=1$.
\par (2.3) For singular type $\textrm{IV}$, we can choose a resolution $\pi:\hat X\to X$ such that
$$\hat X_0=E_1+E_2+E_3+3E_4$$
and
$$K_{\hat X}=\pi^*K_X+E_4.$$
Therefore, we have $\beta=\frac{2}{3}$ and $N=1$.
\par (2.4) For singular type $\textrm{mI}_1$, we can choose a resolution $\pi:\hat X\to X$ such that
$$\hat X_0=mE_1+2mE_2$$
and
$$K_{\hat X}=\pi^*K_X+E_2.$$
Note that $\frac{k_1+1}{a_1}=\frac{k_2+1}{a_2}=\frac{1}{m}$ and $E_1$ and $E_2$ have non-empty intersections. Therefore, we have $\beta=\frac{1}{m}$ and $N=2$.
\par (2.5) For singular type $\textrm{mI}_2$, we can choose a resolution $\pi:\hat X\to X$ such that
$$\hat X_0=mE_1+mE_2+2mE_3+2mE_4.$$
and
and
$$K_{\hat X}=\pi^*K_X+E_3+E_4.$$
Note that $\frac{k_j+1}{a_j}\equiv\frac{1}{m}$, $j=1,2,3,4$, and there exist two components that has non-empty intersections. Therefore, we have $\beta=\frac{1}{m}$ and $N=2$.\\

The above results coincide with \cite[Lemma 3.4]{ST06} and \cite[Section 3.3, Table 1 in page 377]{He}, as expected.

\subsection{Limiting singular metrics on Riemann surfaces}\label{subs3.1} We now apply Theorem \ref{prop3.1} to understand the limiting singular metrics of the K\"ahler-Ricci flow and the continuity method on Riemann surfaces.
\begin{proof}[Proofs of Theorems \ref{theorem1} and \ref{theorem1.1}]
The proofs of Theorems \ref{theorem1} and \ref{theorem1.1} are identical. We only discuss Theorem \ref{theorem1}. Firstly, since $dim(Y)=1$, combining equations \eqref{equation} and \eqref{equation0} gives
\begin{equation}\label{equation2}
\omega_Y=\frac{1}{nV_0}e^{\psi}f_*\Omega.
\end{equation}
We know $Y$ is smooth and $\omega_Y$ is smooth on $Y\setminus S$. Moreover, $\psi$ is a bounded function on $Y$. Then, the asymptotics of $\omega_Y$ near an $s_l\in S$ is just the asymptotics of $f_*\Omega$. Therefore, we can apply Theorem \ref{prop3.1} to conclude \eqref{est1.1}, which implies the metric completion $(X_\infty,d_\infty)$ of $(Y\setminus S,\omega_Y)$ is a compact length metric space homeomorphic to $Y$ (see e.g. \cite[Proposition 3.3]{ZyZz1} for an argument).
\par Theorem \ref{theorem1} is proved.
\end{proof}

\section{Proofs of Theorems \ref{result0}, \ref{result0.1} and \ref{result0.2}}\label{proof.1}

Now, we are able to prove Theorems \ref{result0}, \ref{result0.1} and \ref{result0.2}.

\begin{proof}[Proofs of Theorems \ref{result0}, \ref{result0.1} and \ref{result0.2}]
The proofs of Theorems \ref{result0}, \ref{result0.1} and \ref{result0.2} are identical. We only discuss Theorem \ref{result0}.
\par Having Theorem \ref{theorem1}, Proposition \ref{observation2} and a uniform lower bound for Ricci curvature, one can use the same arguments in \cite[Section 3]{ZyZz1} to conclude Theorem \ref{result0}. To make this paper more readable, we give a sketch as follows.

\par We split the proof into several lemmas. Firstly, we assume without loss of generality that $S=\{s_1\}$ and denote $Y_{reg}:=Y\setminus\{s_1\}$, $X_{reg}:=f^{-1}(Y_{reg})$, $B_\delta:=B_\chi(s_1,\delta)=\{s\in Y|d_\chi(s,s_1)\le\delta\}$, $Y_\delta=Y\setminus B_\delta$ and $X_{\delta}:=f^{-1}(Y_{\delta})$. We assume without loss of generality that, for sufficiently small $\delta$,  $B_\delta$ is the standard disc in $\mathbb{C}$. According to the asymptotic of $\omega_{Y}$ near $s_1$ obtained in Theorem \ref{theorem1}, we can fix a sufficiently large constant $L$ (we are free to increase $L$ if necessary) and a sufficiently small constant $\epsilon_0$ such that for any $\epsilon\le\epsilon_0$ we have
\begin{equation}\label{g1}
diam_{d_\infty}(B_{\epsilon^L})\le\frac{\epsilon}{2}.
\end{equation}
By the asymptotics in Theorem \ref{theorem1}, we furthermore have
\begin{lem}\label{geod1}(see proof of \cite[Lemma 3.6]{ZyZz1})
For any $s,r\in Y_{\epsilon^L}$, there exists a piecewise smooth curve $\gamma\subset Y_{\epsilon^L}$ connecting $s,r$ such that
$$L_{\omega_{Y}}(\gamma)\le d(s,r)+2\epsilon.$$
\end{lem}

\begin{lem}\label{geod2}(see proof of \cite[Lemma 3.6]{ZyZz1})
For any $\epsilon>0$, there exists a $T_\epsilon>0$ such that for any $x,y\in X_{\epsilon^L}$, we can find a piecewise smooth curve $\tau\subset X_{\epsilon^L}$ connecting $x,y$ such that for any $t\in[T_\epsilon,\infty)$,
$$L_{\omega(t)}(\tau)\le d(f(x),f(y))+3\epsilon.$$
\end{lem}

\begin{lem}\label{geod3}(see proof of \cite[Lemma 3.6]{ZyZz1})
For any $\epsilon>0$, there exists a $T_\epsilon>0$ such that for any $x,y\in X_{\epsilon}$ and $t\in[T_\epsilon,\infty)$, there exists a piecewise smooth curve $\tau^t\subset X_{\epsilon}$ with
$$L_{\omega(t)}(\tau^t)\le d_{\omega(t)}(x,y)+4\epsilon.$$
\end{lem}

Now we can obtain the diameter bound of $(X,\omega(t))$ by using Ricci curvature lower bound.
\begin{lem}\label{geod4}\cite[Lemma 3.7]{ZyZz1}
There exists a constant $D>0$ such that for any $t\in[1,\infty)$,
$$diam_{\omega(t)}(X)\le D.$$
\end{lem}
\begin{proof}
To see the role of lower bound for Ricci curvature, we contain some details. Note that the compactness of $(\Sigma,d)$ and Lemma \ref{geod2} imply that, for sufficiently small $\epsilon$, there exists two constants $T$ and $D_1>0$ such that for all $t\ge T$
\begin{equation}\label{g4}
diam_{d_{\omega(t)}}(X_{\epsilon^L})\le D_1.
\end{equation}
Using Proposition \ref{facts}(1) and the fact that $X\setminus X_{reg}$ has real codimension $\ge2$ ($X\setminus X_{reg}$ is in fact a proper subvariety of $X$), up to possibly decreasing $\epsilon$ and increasing $L$, for all $t\ge T$ we have,
\begin{equation}
Vol_{\omega(t)}(X\setminus X_{\epsilon^L})\le e^{-(n-1)t}\epsilon.
\end{equation}
Let $x_t\in X\setminus X_{\epsilon^L}$ be a point achieves the maximal distance $R_t$ to $X_{\epsilon^L}$ in $(X,\omega(t))$, i.e.,
$$R_t=\sup_{x\in X\setminus X_{\epsilon^L}}\inf d_{\omega(t)}(x,X_{\epsilon^L})=\sup_{x\in X\setminus X_{\epsilon^L}}\inf_{y\in X_{\epsilon^L}} d_{\omega(t)}(x,y).$$
Note that $B_{\omega(t)}(x_t,R_t)\subset X\setminus X_{\epsilon}$ and $X_{\epsilon^L}\subset B_{\omega(t)}(x_t,R_t+D_1)$. On the one hand, for some fixed constant $C_0\ge1$ we have
\begin{equation}\label{g5}
\frac{Vol_{\omega(t)}(B_{\omega(t)}(x_t,R_t+D_1))}{Vol_{\omega(t)}(B_{\omega(t)}(x_t,R_t))}\ge\frac{Vol_{\omega(t)}(X_{\epsilon^L})}{Vol_{\omega(t)}(X\setminus X_{\epsilon^L})}\ge \frac{C_0^{-1}e^{-(n-1)t}-\epsilon e^{-(n-1)t}}{\epsilon e^{-(n-1)t}}\ge \frac{1}{2C_0}\epsilon^{-1}.
\end{equation}
On the other hand, by the lower bound for Ricci curvature (assume without loss of generality $Ric(\omega(t))\ge-(2n-1)\omega(t)$) and volume comparison we have
\begin{equation}\label{g6}
\frac{Vol_{\omega(t)}(B_{\omega(t)}(x_t,R_t+D_1))}{Vol_{\omega(t)}(B_{\omega(t)}(x_t,R_t))}\le\frac{\int_0^{R_t+D_1}sinh^{2n-1}vdv}{\int_{0}^{R_t}sinh^{2n-1}vdv},
\end{equation}
where the right hand side of \eqref{g6} will converges to $1$ if $R_t\to\infty$. Combining \eqref{g5}, we know $R_t$ is uniform bounded from above for all $t\ge T$. But by triangle inequality we have
$$diam_{\omega(t)}(X)\le2R_t+D_1,$$
which implies the desired diameter upper bound for $(X,\omega(t))$.
\par Lemma \ref{geod4} is proved.
\end{proof}

\begin{lem}\label{geod5}\cite[Lemma 3.8]{ZyZz1}
There exist two constants $\epsilon>0$ and $T>0$ such that, after possibly increasing $L$, for all $\epsilon\le\epsilon_0$ and $t\ge T$ we have
$$diam_{\omega(t)}(X\setminus X_{\epsilon^L})\le6\epsilon.$$
\end{lem}
\begin{proof}
Let $R_t$ be the same as in the proof Lemma \ref{geod4}. Since by Lemma \ref{geod4} $R_t$ is bounded from above, the arguments in Lemma \ref{geod4} further imply that, after possibly increasing $L$ to $L^{2n+1}$, there exist constant $T>0$ and $C_1,C_2\ge1$ such that for all $t\ge T$,
$$C_1^{-1}\epsilon^{-(2n+1)}\le\frac{\int_0^{R_t+D_1}sinh^{2n-1}vdv}{\int_{0}^{R_t}sinh^{2n-1}vdv}\le \frac{C_2}{\int_{0}^{R_t}sinh^{2n-1}vdv},$$
so,
$$\int_{0}^{R_t}sinh^{2n-1}vdv\le C_1C_2\epsilon^{2n+1},$$
which implies
$$R_t\le\epsilon$$
Now for any $x\in X\setminus X_{\epsilon^L}$, we choose a $\bar x\in \partial X_{\epsilon^L}=f^{-1}(\partial B_{\epsilon^L})$ with $d_{\omega(t)}(x,\bar x)\le R_t\le\epsilon$. Assume we are given arbitrary two points $x,y\in X\setminus X_{\epsilon^L}$ and the corresponding points $\bar x,\bar y$. By Lemma \ref{geod2} and \eqref{g1} we have
$$d_{\omega(t)}(\bar x,\bar y)\le d(f(\bar x),f(\bar y))+3\epsilon\le4\epsilon.$$
By triangle inequality we have
$$d_{\omega(t)}(x,y)\le d_{\omega(t)}(x,\bar x)+d_{\omega(t)}(\bar x,\bar y)+d_{\omega(t)}(\bar y,y)\le6\epsilon.$$
Lemma \ref{geod5} is proved.
\end{proof}

Now we are ready to prove Theorem \ref{result0}.\\

\emph{Completion of proof of Theorem \ref{result0}.}
Define a map $g:Y\to X$ by choosing $g(s)\in X_s=f^{-1}(s)$ for every $s\in Y$. By Definitions, it suffices to show that for any small $\epsilon>0$, there exists a $T_{\epsilon}>0$ such that for all $t\ge T_{\epsilon}$, the followings hold.
\begin{itemize}
\item[(1)] $|d_{\omega(t)}(x,y)-d_\infty(f(x),f(y))|\le8\epsilon$ for all $x,y\in X$;
\item[(2)] $|d_{\omega(t)}(g(s),g(r))-d_\infty(s,r)|\le8\epsilon$ for all $s,r\in Y$;
\item[(3)] $d_{\omega(t)}(x,g(f(x)))\le8\epsilon$ for all $x\in X$;
\item[(4)] $d_\infty(s,f(g(s)))\le8\epsilon$ for all $s\in Y$.
\end{itemize}
We firstly note that $f(g(s))=s$ for all $s\in Y$ and hence item (4) holds trivially. \\
\underline{Proof of item (3)}: If $x\in X_{\epsilon^L}$, item (3) follows from the uniform collapsing of smooth fibers over $Y_{\epsilon}$ implied by Proposition \ref{facts}(3); if $x\in X\setminus X_{\epsilon^L}$, since $g(f(x))\in X\setminus X_{\epsilon^L}$, item (3) follows from Lemma \ref{geod5}.
\par The proofs of items (1) and (2) are same. Here we only prove item (1).\\
\underline{Proof of item (1)}: Assume we are given two arbitrary points $x,y\in X$.\\
\emph{Case I}: $x,y\in X_{\epsilon^L}$. In this case, by Lemma \ref{geod2} we see
\begin{equation}\label{g7}
d_{\omega(t)}(x,y)\le d_\infty(f(x),f(y))+3\epsilon.
\end{equation}
On the other hand, by Lemma \ref{geod3} we fix a piecewise smooth curve $\tau^t\subset X_{\epsilon^L}$ connecting $x,y$ with
$$L_{\omega(t)}(\tau^t)\le d_{\omega(t)}(x,y)+4\epsilon.$$
Since $\tau^t\subset X_{\epsilon^L}$, we apply Proposition \ref{observation2} to find a constant $T_{\epsilon}>0$ such that for all $t\ge T_{\epsilon}$ we have
$$L_{\omega_{Y}}(f(\tau^t))\le L_{\omega(t)}(\tau^t)+\epsilon$$
which implies
\begin{equation}\label{g8}
d_\infty(f(x),f(y))\le L_{\omega_{Y}}(f(\tau^t))\le d_{\omega(t)}(x,y)+4\epsilon.
\end{equation}
Combining \eqref{g7} and \eqref{g8}, we obtain item (1) in this case.\\
\emph{Case II}: $x\in X\setminus X_{\epsilon^L}$, $y\in X_{\epsilon}$. We fix a $\bar x\in \partial X_{\epsilon}=f^{-1}(\partial B_{\epsilon^L})$ with $d_{\omega(t)}(x,\bar x)\le\epsilon$ as in the proof of Lemma \ref{geod5}. Then by Case I we know
\begin{equation}\label{g9}
|d_{\omega(t)}(\bar x,y)-d_\infty(f(\bar x),f(y))|\le4\epsilon.
\end{equation}
So, by the triangle inequality and the fact that $d_\infty(f(x),f(\bar x))\le\frac{\epsilon}{2}$ by \eqref{g1} we have
\begin{align}\label{g10}
d_{\omega(t)}(x,y)&\le d_{\omega(t)}(x,\bar x)+d_{\omega(t)}(\bar x,y)\nonumber\\
&\le\epsilon+d_\infty(f(\bar x),f(y))+4\epsilon\nonumber\\
&\le d_\infty(f(x),f(y))+d_\infty(f(\bar x),f(x))+5\epsilon\nonumber\\
&\le d_\infty(f(x),f(y))+6\epsilon
\end{align}
and similarly,
\begin{equation}\label{g11}
d_\infty(f(x),f(y))\le d_{\omega(t)}(x,y)+6\epsilon
\end{equation}
Combining \eqref{g10} and \eqref{g11}, we obtain item(1) in this case.\\
\emph{Case III}: $x,y\in X\setminus X_{\epsilon^L}$. By \eqref{g1} and Lemma \ref{geod5} we know
$$d_{\omega(t)}(x,y)+d_\infty(f(x),f(y))\le7\epsilon,$$
from which item (1) follows in this case.
\par Combining Cases I-III, we have proved item (1).
\par Theorem \ref{result0} is proved.

\end{proof}

\section{Proofs of Theorems \ref{result} and \ref{result.1}}\label{proof}

This section contains proofs of Theorems \ref{result} and \ref{result.1}. Since the arguments are very similar, we will only discuss Theorem \ref{result}. The proof of Theorem \ref{result} needs more efforts than Theorem \ref{result0}. An outline is as follows. Firstly, we use Theorem \ref{theorem1} to determine the asymptotics of the generalized K\"ahler-Einstein current on $\tilde\Sigma$; secondly, under the assumption that Ricci curvature is uniformly bounded from below, we modify discussions in Section \ref{proof.1} to obtain a uniform diameter upper bound for the K\"ahler-Ricci flow on $\tilde X$; finally we apply some results and arguments of Cheeger-Colding \cite{CC1,CC2} and Gross-Tosatti-Zhang \cite{GTZ13,GTZ,ToZy17} to prove Gromov-Hausdorff convergence.

\subsection{Generalized K\"ahler-Einstein current on $\tilde\Sigma$}\label{subs3.1'} Recall we are in the Setup \ref{setup3}. We can also fix a smooth positive volume form $\Omega'$ on $X'$ with $\sqrt{-1}\partial\bar\partial\log\Omega'=f'^*\chi'$. Then $\Sigma'$ admits a unique generalized K\"ahler-Einstein current $\omega_{\Sigma'}=\chi'+\sqrt{-1}\partial\bar\partial\psi'$ solved by
\begin{equation}\label{GKE0}
\omega_{\Sigma'}=\chi'+\sqrt{-1}\partial\bar\partial\psi'=m^{-1}e^{\psi'}f'_*\Omega'.
\end{equation}
In the following, we identify $c_1(X)$ on $X$ with its pullback on $\tilde X$, $\chi$ on $\Sigma$ with its pullback on $\tilde\Sigma$ and so on. Note that $\tilde S:=(S\times\Sigma')\cup (\Sigma\times S')$ has simple normal crossing support.
\par According to \cite{ST06,ST12}, the generalized K\"ahler-Einstein current $\omega_{\tilde\Sigma}$ on $\tilde\Sigma$ is solved by
\begin{equation}\label{GKE1}
(\chi+\chi'+\sqrt{-1}\partial\bar\partial\tilde\psi)^2=\frac{1}{\binom{n+m}{2}}e^{\tilde\psi}\tilde f_*(\Omega\wedge\Omega').
\end{equation}
Meanwhile, easily we have
$$\tilde f_*(\Omega\wedge\Omega')=f_*\Omega\wedge f'_*\Omega'.$$
Therefore, it is not hard to see that the unique solution $\tilde\psi$ to \eqref{GKE1} is given by
$$\tilde\psi=\psi+\psi'+c_{n,m},$$
where $\psi'$ is the unique solution to \eqref{GKE0} and $c_{n,m}:=\log\frac{(n+m)(n+m-1)}{nm}$. Consequently,
\begin{align}\label{GKE2}
\omega_{\tilde\Sigma}&=\chi+\chi'+\sqrt{-1}\partial\bar\partial\tilde\psi\nonumber\\
&=\omega_{\Sigma}+\omega_{\Sigma'}.
\end{align}

By Theorem \ref{theorem1} we know $\omega_{\tilde\Sigma}$ on $\tilde\Sigma$ near point in $(\Sigma\setminus S)\times S'$ is locally equivalent to
$$\sqrt{-1}\left(ds\wedge d\bar s+|s'|^{-2(1-\beta')}(-\log|s'|)^{N'}ds'\wedge d\bar {s'}\right),$$
while near point in $S\times S'$ is locally equivalent to
$$\sqrt{-1}\left(|s|^{-2(1-\beta)}(-\log|s|)^{N}ds\wedge d\bar s+|s'|^{-2(1-\beta')}(-\log|s'|)^{N'}ds'\wedge d\bar {s'}\right).$$

From now on, we assume without loss of generality that $S=\{s_1\}$ and $S'=\{s_1'\}$. Denote $\Sigma_{reg}=\Sigma\setminus\{s_1\}$, $\Sigma_{reg}'=\Sigma'\setminus\{s_1'\}$ and $\tilde\Sigma_{reg}=\Sigma_{reg}\times\Sigma_{reg}'$.
\par Define a metric $\tilde d_\infty$ on $\tilde\Sigma$ as follows. For any $s,r\in\Sigma_{reg}$ and $s'\in\Sigma_{reg}'$, set
\begin{itemize}
\item[(1)]$\tilde d_\infty=d_{\tilde\omega_{\tilde\Sigma}}$ on $\Sigma_{reg}\times\Sigma_{reg}'$;
\item[(2)]$\tilde d_\infty((s,s_1'),(r,s_1'))=d_{\omega_{\Sigma}}(s,r)$;
\item[(3)]$\tilde d_\infty((s_1,s_1')(s,s_1'))=\lim_{k\to\infty}d_{\omega_{\Sigma}}(s_{1k},s)$ for sequence $s_{1k}\in\Sigma_{reg}$ converges to $s_1$a;
\item[(4)]$\tilde d_\infty((s,s_1'),(r,s'))=\lim_{k\to\infty}d_{\omega_{\tilde\Sigma}}((s,s_{1k}'),(r,s'))$ for a sequence $s_{1k}'\in\Sigma_{reg}'$ converges to $s_1'$;
\item[(5)]$\tilde d_\infty((s_1,s_1'),(s,s'))=\lim_{k\to\infty}d_{\omega_{\tilde\Sigma}}((s_{1k},s_{1k}'),(s,s'))$ for a sequence $s_{1k}\in\Sigma_{reg}$ converges to $s_1$ and a sequence $s_{1k}'\in\Sigma_{reg}'$ converges to $s_1'$;
\item[(6)]In other cases, we define $\tilde d_\infty$ by interchanging the role of $\Sigma$ and $\Sigma'$.
\end{itemize}
The above definition does not depend on the choice of approximating sequence $s_{1k},s_{1k}'$. For example, in case (5), if we use different sequences $s_{1l},s_{1l}'$, then
\begin{align}
&|d_{\omega_{\tilde\Sigma}}((s_{1k},s_{1k}'),(s,s'))-d_{\omega_{\tilde\Sigma}}((s_{1l},s_{1l}'),(s,s'))|\nonumber\\
&\le d_{\omega_{\tilde\Sigma}}((s_{1k},s_{1k}'),(s_{1l},s_{1l}'))\nonumber\\
&\le d_{\omega_{\Sigma}}(s_{1k},s_{1l})+d_{\omega_{\Sigma'}}(s_{1k}',s_{1l}')\nonumber\\
&\to0\nonumber,
\end{align}
where in the last inequality we have used Cauchy inequality for product manifold with product Riemannian metric and the last convergence takes place uniformly as $k,l\to\infty$.
\par In conclusion, $(\tilde\Sigma,\tilde d_\infty)$ is a compact length metric space and is the metric completion of $(\tilde\Sigma_{reg},\omega_{\tilde\Sigma})$. We will see that $(\tilde\Sigma,\tilde d_\infty)$ is the unique limiting space in Theorem \ref{result}.

\subsection{Diameter bound of $\omega(t)$}
Denote $B_{\delta}:=B_{\chi}(s_1,\delta)=\{s\in\Sigma|d_{\chi}(s,s_1)\le\delta\}$, $B'_{\delta}:=B_{\chi'}(s_1',\delta)$, $\Sigma_\delta:=\Sigma\setminus B_{\delta}$, $\Sigma_{\delta}':=\Sigma'\setminus B_{\delta}'$, $X_\delta:=f^{-1}(\Sigma_\delta)$, $\tilde X_{\delta}:=f'^{-1}(\Sigma_{\delta}')$.
\par According to the asymptotics of $\omega_{\tilde\Sigma}$, we fix a sufficiently large $L$ (we are free to increase $L$ if necessary) such that for any $\epsilon<\frac{1}{2}$, we have
\begin{equation}\label{diamonbase}
diam_{\tilde d_\infty}(B_{\epsilon^L}\times B_{\epsilon^L}')\le diam_{d_\infty}(B_{\epsilon^L})+diam_{d_\infty'}(B_{\epsilon^L}')\le\frac{\epsilon}{4}
\end{equation}
and
\begin{equation}\label{diamonbase1}
L_{d_\infty}(\partial B_{\epsilon^L})+L_{d_\infty'}(\partial B_{\epsilon^L}')\le\frac{\epsilon}{4}.
\end{equation}
The above first inequality in \eqref{diamonbase} holds since $\tilde d_\infty$ is induced by product metric $\omega_{GKE}+\omega_{GKE}'$ on regular part. In particular, we have
\begin{equation}\label{GH1}
d_{GH}((\tilde\Sigma,\tilde d_\infty),(\Sigma_{\epsilon^L}\times\Sigma_{\epsilon^L}',\tilde d_\infty))\le\frac{\epsilon}{4}.
\end{equation}
We need the following
\begin{lem}\label{lemonbase}
For any points $(s,s'),(r,r')\in\Sigma_{\epsilon^L}\times\Sigma_{\epsilon^L}'$, we can choose a piecewise smooth curve $\tilde\gamma\subset\Sigma_{\epsilon^L}\times\Sigma_{\epsilon^L}'$ connecting $(s,s')$ and $(r,r')$ with
\begin{equation}\label{diamonbase3.1}
L_{\tilde d}(\tilde\gamma)\le\tilde d((s,s'),(r,r'))+\epsilon.
\end{equation}
\end{lem}
\begin{proof}
Firstly, we fix a curve $\tilde\gamma_0=(\gamma_0,\gamma_0'):[0,1]\to\Sigma_{reg}\times\Sigma_{reg}'$ such that
\begin{equation}
L_{\tilde d}(\tilde\gamma_0)=L_{\omega_{\tilde\Sigma}}(\tilde\gamma_0)\le \tilde d((s,s'),(r,r'))+\frac{\epsilon}{4}.
\end{equation}
We assume without loss of generality that $\gamma_0\cap B_{\epsilon^L}\neq\emptyset$ and $\gamma_0'\cap B_{\epsilon^L}'\neq\emptyset$. Set $u_1:=\sup\{u\in(0,1)|\gamma_0([0,u))\subset\Sigma_{\epsilon^L}\}$ and $u_2:=\inf\{u\in(0,1)|\gamma_0((u,1])\subset\Sigma_{\epsilon^L}\}$. Define $u_1',u_2'$ for $\gamma_0'$ similarly. Define $\tilde\gamma=(\gamma,\gamma')$, where $\gamma$ is obtained from $\gamma_0$ by changing $\gamma_0|_{[u_1,u_2]}$ to a smooth curve which lies on $\partial B_{\epsilon^L}$ and connects $\gamma_0(u_1)$ and $\gamma_0(u_2)$. $\gamma'$ is obtained from $\gamma_0'$ similarly. It suffices to show $\tilde\gamma$ satisfies \eqref{diamonbase3.1}. We separate our arguments into several cases as follows.
\par \emph{Case (1): $u_2\le u_1'$.} In this case,
\begin{align}
L_{\tilde d_\infty}(\tilde\gamma)-L_{\tilde d_\infty}(\tilde\gamma_0)&=L_{\omega_{\tilde\Sigma}}(\tilde\gamma)-L_{\omega_{\tilde\Sigma}}(\tilde\gamma_0)\nonumber\\
&=I+II\nonumber,
\end{align}
where
\begin{align}\label{I1}
I&=\int_{u_1}^{u_2}\sqrt{\omega_{\Sigma}(\dot{\gamma},\dot{\gamma})+\omega_{\Sigma'}(\dot{\gamma'},\dot{\gamma'})}du-\int_{u_1}^{u_2}\sqrt{\omega_{\Sigma}(\dot{\gamma}_0,\dot{\gamma}_0)+\omega_{\Sigma'}(\dot{\gamma'}_0,\dot{\gamma'}_0)}du\nonumber\\
&\le\int_{u_1}^{u_2}\sqrt{\omega_{\Sigma}(\dot{\gamma},\dot{\gamma})+\omega_{\Sigma'}(\dot{\gamma'},\dot{\gamma'})}du-\int_{u_1}^{u_2}\sqrt{\omega_{\Sigma}(\dot{\gamma'}_0,\dot{\gamma'}_0)}du\\
&\le\int_{u_1}^{u_2}\sqrt{\omega_{\Sigma}(\dot{\gamma},\dot{\gamma})}du+\int_{u_1}^{u_2}\sqrt{\omega_{\Sigma'}(\dot{\gamma'},\dot{\gamma'})}du-\int_{u_1}^{u_2}\sqrt{\omega_{\Sigma'}(\dot{\gamma'}_0,\dot{\gamma'}_0)}du\nonumber\\
&=\int_{u_1}^{u_2}\sqrt{\omega_{\Sigma}(\dot{\gamma},\dot{\gamma})}du\nonumber\\
&\le L_{d_\infty}(\partial B_{\epsilon^L})\le\frac{\epsilon}{4}\nonumber,
\end{align}
and
\begin{align}
II&=\int_{u_1'}^{u_2'}\sqrt{\omega_{\Sigma}(\dot{\gamma},\dot{\gamma})+\omega_{\Sigma'}(\dot{\gamma'},\dot{\gamma'})}du-\int_{u_1'}^{u_2'}\sqrt{\omega_{\Sigma}(\dot{\gamma}_0,\dot{\gamma}_0)+\omega_{\Sigma'}(\dot{\gamma'}_0,\dot{\gamma'}_0)}du\nonumber\\
&\le\int_{u_1'}^{u_2'}\sqrt{\omega_{\Sigma}(\dot{\gamma},\dot{\gamma})+\omega_{\Sigma'}(\dot{\gamma'},\dot{\gamma'})}du-\int_{u_1'}^{u_2'}\sqrt{\omega_{\Sigma}(\dot{\gamma}_0,\dot{\gamma}_0)}du\nonumber\\
&\le\int_{u_1'}^{u_2'}\sqrt{\omega_{\Sigma}(\dot{\gamma},\dot{\gamma})}du+\int_{t_1'}^{t_2'}\sqrt{\omega_{\Sigma'}(\dot{\gamma'},\dot{\gamma'})}du-\int_{u_1'}^{u_2'}\sqrt{\omega_{\Sigma}(\dot{\gamma}_0,\dot{\gamma}_0)}du\nonumber\\
&=\int_{u_1'}^{u_2'}\sqrt{\omega_{\Sigma'}(\dot{\gamma'},\dot{\gamma'})}du\nonumber\\
&\le L_{d_\infty'}(\partial B_{\epsilon^L}')\le\frac{\epsilon}{4}\nonumber.
\end{align}
Therefore,
\begin{align}
L_{\tilde d_\infty}(\tilde\gamma)&\le L_{\tilde d_\infty}(\tilde\gamma_0)+\frac{\epsilon}{2}\nonumber\\
&\le\tilde d_\infty((s,s'),(r,r'))+\frac{3\epsilon}{4}\nonumber.
\end{align}
\par \emph{Case (2): $u_1=u_1'$ and $u_2=u_2'$.} In this case,
\begin{align}
&L_{\tilde d_\infty}(\tilde\gamma)-L_{\tilde d_\infty}(\tilde\gamma_0)\nonumber\\
&=L_{\omega_{\tilde\Sigma}}(\tilde\gamma)-L_{\omega_{\tilde\Sigma}}(\tilde\gamma_0)\nonumber\\
&=\int_{u_1}^{u_2}\sqrt{\omega_{\Sigma}(\dot{\gamma},\dot{\gamma})+\omega_{\Sigma'}(\dot{\gamma'},\dot{\gamma'})}du-\int_{u_1}^{u_1'}\sqrt{\omega_{\Sigma}(\dot{\gamma}_0,\dot{\gamma}_0)+\omega_{\Sigma'}(\dot{\gamma'}_0,\dot{\gamma'}_0)}du\nonumber\\
&\le\int_{u_1}^{u_2}\sqrt{\omega_{\Sigma}(\dot{\gamma},\dot{\gamma})+\omega_{\Sigma'}(\dot{\gamma'},\dot{\gamma'})}du\nonumber\\
&\le\int_{u_1}^{u_2}\sqrt{\omega_{\Sigma}(\dot{\gamma},\dot{\gamma})}du+\int_{u_1}^{u_2}\sqrt{\omega_{\Sigma'}(\dot{\gamma'},\dot{\gamma'})}du\nonumber\\
&\le L_{d_\infty}(\partial B_{\epsilon^L})+L_{d_\infty'}(\partial B_{\epsilon^L}')\nonumber\\
&\le\frac{\epsilon}{4}\nonumber,
\end{align}
\par \emph{Case(3): $u_1\le u_1'\le u_2\le u_2'$ or $u_1\le u_1'\le u_2'\le u_2$.} These cases can be checked by a combination of arguments in the above two  cases.
\par \emph{Case (4)}. The remaining cases can be checked by interchanging the role of $(u_1,u_2)$ and $(u_1',u_2')$.
\par Lemma \ref{lemonbase} is proved.
\end{proof}

Now we check an analog of \cite[Lemma 3.6]{ZyZz1}.
\begin{lem}\label{lemonbase1}
For any $\epsilon>0$, there exists a $T_\epsilon>0$ such that for any $(x,x'),(y,y')\in X_{\epsilon^L}\times X_{\epsilon^L}'$ and $t\in[T_\epsilon,\infty)$, we can choose a piecewise smooth curve $\tilde\tau\subset X_{\epsilon^L}\times X_{\epsilon^L}'$ connecting $(x,x')$ and $(y,y')$ with
\begin{equation}
L_{\omega(t)}(\tilde\tau)\le\tilde d_\infty((f(x),f'(x)),(f(y),f'(y')))+2\epsilon.
\end{equation}
\end{lem}
\begin{proof}
By Lemma \ref{lemonbase}, we can choose a piecewise smooth curve $\tilde\gamma\subset\Sigma_{\epsilon^L}\times\Sigma_{\epsilon^L}'$ connecting $(f(x),f'(x'))$ and $(f(y),f'(y'))$ with
\begin{equation}\label{diamonbase4}
L_{\tilde d_\infty}(\tilde\gamma)=L_{\omega_{\tilde\Sigma}}(\tilde\gamma)\le\tilde d_\infty(f(x),f'(x'),(f(y),f'(y')))+\epsilon.
\end{equation}
Then one easily lifts $\tilde\gamma$ by $\tilde f$ to a curve $\tilde\tau\subset X_{\epsilon^L}\times X_{\epsilon^L}'$ by the manner in \cite[Lemma 3.6]{ZyZz1}, which, by combining items (4) and (5) in Proposition \ref{facts} and choosing sufficiently large $T_\epsilon$, satisfies
\begin{align}
L_{\omega(t)}(\tilde\tau)&\le L_{\tilde f^*\omega_{\tilde\Sigma}}(\tilde\tau)+\epsilon\nonumber\\
&=L_{\omega_{\tilde\Sigma}}(\tilde\gamma)+\epsilon\nonumber\\
&\le\tilde d_\infty(f(x),f'(x'),(f(y),f'(y')))+2\epsilon\nonumber.
\end{align}
Lemma \ref{lemonbase1} is proved.
\end{proof}
Now, by combining Lemma \ref{lemonbase1}, item (1) of Proposition \ref{facts} and the assumption of Ricci curvature lower bound, we can repeat arguments in \cite[Lemmas 3.7, 3.8]{ZyZz1} to conclude
\begin{lem}\label{lemondiam}
Assume the Ricci curvature of $\omega(t)$ is uniformly bounded from below.
\begin{itemize}
\item[(1)] There exist two positive constant $C$ and $T$ such that for all $t\in[T,\infty)$,
$$diam_{\omega(t)}(\tilde X)\le C.$$
\item[(2)] For any $\epsilon\le\frac{1}{2}$, set $R_t$ be the maximal distance with respect to $(\tilde X,\omega(t))$ from points in $\tilde X\setminus(X_{\epsilon^L}\times X_{\epsilon^L}')$ to $X_{\epsilon^L}\times X_{\epsilon^L}'$. Then up to possibly increasing $L$, there exists some $T_\epsilon>0$ such that for all $t\in[T_\epsilon,\infty)$ we have
   $$R_t\le\epsilon.$$
\end{itemize}
In particular, for all $t\in[T_\epsilon,\infty)$,
\begin{equation}\label{GH2}
d_{GH}((\tilde X,d_{\omega(t)}),(X_{\epsilon^L}\times X_{\epsilon^L}',d_{\omega(t)}))\le\epsilon.
\end{equation}
\end{lem}

\subsection{Gromov-Hausdorff convergence}
Since $(\tilde X,\omega(t))$ have uniform Ricci curvature lower bound and diameter upper bound, we now make use of Proposition \ref{facts}, Lemma \ref{L1conv} and some results and arguments in \cite{CC1,CC2,GTZ13,GTZ,ToZy17} to prove Gromov-Hausdorff convergence.
\par By Gromov's precompactness theorem, for any given time sequence we can choose a subsequence $t_k\to\infty$ and a compact length metric space $(Z,d_Z)$ with $(\tilde X,\omega(t_k))\to(Z,d_Z)$ in Gromov-Hausdorff topology, as $t_k\to\infty$. As we have the estimates in Proposition \ref{facts}, the following lemma can be checked by the same arguments in \cite[Lemma 5.1]{GTZ13}.
\begin{lem}
There exist an open subset $Z_0$ of $Z$ and a local isometric homeomorphism $h:(\tilde\Sigma_{reg},\tilde d)\to (Z_0,d_Z)$.
\end{lem}
Now we recall the construction in \cite[Section 1]{CC1} of a Radon measure $\nu$ on $Z$, i.e., the renormalized limiting measure, as follows. Fix a $ \underline{z}\in Z_0$ and $\underline{\tilde p_k}\in\tilde X$ with $\underline{\tilde p_k}\to\underline{z}$ under Gromov-Hausdorff convergence of $(\tilde X,\omega(t_k))\to(Z,d_Z)$. For any given $\tilde x\in\tilde X$ and $r>0$, let
$$\underline{V}_k(\tilde x,r)=\frac{Vol_{\omega(t_k)}(B_{\omega(t_k)}(\tilde x,r))}{Vol_{\omega(t_k)}(B_{\omega(t_k)}(\underline{\tilde p_k},1))}.$$
Then \cite[Theorem 1.6]{CC1} gives a continuous function
$$\underline{V}_{\infty}:Z\times[0,\infty)\to[0,\infty)$$
with, for any given $\tilde x_k\to z$ under Gromov-Hausdorff convergence $(\tilde X,\omega(t_k))\to(Z,d_Z)$ and $r>0$, $\underline{V}_{k}(\tilde x_k,r)\to\underline{V}_\infty(z,r)$. Moreover, by \cite[Theorem 1.10]{CC1}, there exists a unique Radon measure $\nu$ on $Z$ such that for all $z\in Z$ and $r>0$, we have
$$\nu(B_{d_Z}(z,r)=\underline{V}_\infty(z,r).$$
By definition of $\underline{V}_k$ and volume comparison theorems, we also have, for any $r\le R$,
\begin{equation}
\frac{\nu(B_{d_Z}(z,r))}{\nu(B_{d_Z}(z,R))}\ge\frac{V_{2n+2m,-1}(r)}{V_{2n+2m,-1}(R)},
\end{equation}
where $V_{n,-1}(r)$ denotes the volume of ball with radii $r$ in the simply connected space form of dimension $n$ and curvature $-1$. For any compact subset $K\subset Z$,
$$\nu(K)=\lim_{\delta\to0}\nu_{\delta}(K)=\lim_{\delta\to0}\inf\left\{\sum_{i}\underline{V}_\infty(z_i,r_i)|K\subset\cup_{i}B_{d_Z}(z_i,r_i),r_i<\delta\right\}.$$
In the following, up to some scaling on $\omega(t_k)$ and $\omega_{\tilde\Sigma}$, we may assume without loss of generality that $B_{\omega_{\tilde\Sigma}}(h^{-1}(\underline{z}),2)\subset\tilde\Sigma_{reg}$ and is geodesically convex, i.e., any two $z_1,z_2\in B_{\omega_{\tilde\Sigma}}(h^{-1}(\underline{z}),2)$ can be connected by a minimal geodesic contained in $B_{\omega_{\tilde\Sigma}}(h^{-1}(\underline{z}),2)$.
The following is an analog of \cite[Lemma 5.2]{GTZ13}.
\begin{lem}\label{meas}
There exist a positive constant $C_1$ such that
$$\underline{V}_{\infty}(z,r)=C_1\int_{\tilde f^{-1}(B_{\omega_{\tilde\Sigma}}(h^{-1}(z),r))}e^{\tilde\psi}\tilde\Omega,$$
whenever $z\in Z_0$ and $r\le1$ with $B_{\omega_{\tilde\Sigma}}(h^{-1}(z),2r)\subset \tilde\Sigma_{reg}$  and is geodesically convex. Recall $\tilde\psi$ is the solution to \eqref{GKE1}.
\end{lem}
\begin{proof}
The proof is almost identical to \cite[Lemma 5.2]{GTZ13} except one minor modification due to the fact that the involved equation is different from \cite{GTZ13}. For convenience, we give some details here. Firstly, by using properties of the K\"ahler-Ricci flow collected in Proposition \ref{facts}, Proposition \ref{observation2} and the same arguments in \cite[Lemma 5.2]{GTZ13}, we can find a positive function $\epsilon(t)$, which converges to zero as $t\to\infty$, such that
$$\tilde f^{-1}(B_{\omega_{\tilde\Sigma}}(h^{-1}(z),r-\epsilon(t_k)))\subset B_{\omega(t_k)}(\tilde x_k,r)\subset \tilde f^{-1}(B_{\omega_{\tilde\Sigma}}(h^{-1}(z),r+\epsilon(t_k))),$$
where $\tilde x_k\in\tilde X$ is some fixed sequence converging to $z$ under Gromov-Hausdorff convergence $(\tilde X,\omega(t_k))\to(Z,d_Z)$. Now using the equation \eqref{conti.eq} and the fact that $e^{\dot\varphi(t)+\varphi(t)}\to e^{f^*\tilde\psi}$ in $L^{1}(\tilde X)$-topology by Lemma \ref{L1conv}, we have
\begin{align}
\underline{V}_k(\tilde x_k,r)&=\frac{Vol_{\omega(t_k)}(B_{\omega(t_k)}(\tilde x_k,r))}{Vol_{\omega(t_k)}(B_{\omega(t_k)}(\underline{\tilde p_k},1))}\nonumber\\
&=\frac{\int_{B_{\omega(t_k)}(\tilde x_k,r)}e^{-(n+m-2)t}e^{\dot\varphi(t_k)+\varphi(t_k)}\tilde\Omega}{\int_{B_{\omega(t_k)}(\underline{\tilde p_k},1)}e^{-(n+m-2)t}e^{\dot\varphi(t_k)+\varphi(t_k)}\tilde\Omega}\nonumber\\
&=\frac{\int_{B_{\omega(t_k)}(\tilde x_k,r)}e^{\dot\varphi_k+\varphi(t_k)}\tilde\Omega}{\int_{B_{\omega(t_k)}(\underline{\tilde p_k},1)}e^{\dot\varphi(t_k)+\varphi(t_k)}\tilde\Omega}\nonumber\\
&\to\frac{\int_{\tilde f^{-1}(B_{\omega_{\tilde\Sigma}}(h^{-1}(z),r))}e^{\tilde\psi}\tilde\Omega}{\int_{\tilde f^{-1}(B_{\omega_{\tilde\Sigma}}(h^{-1}(\underline{z}),1))}e^{\tilde\psi}\tilde\Omega}\nonumber.
\end{align}
Therefore, if we set $C_1=\left(\int_{B_{\omega_{\tilde\Sigma}}(h^{-1}(\underline{z}),1)}e^{\tilde\psi}\tilde\Omega\right)^{-1}$, then the desired result follows.
\end{proof}

\begin{rem}
To carry out the analog of Lemma \ref{meas} for the continuity method \eqref{LT0}, one only needs to replace Lemma \ref{L1conv} by Lemma \ref{L1conv.1}.
\end{rem}

A direct consequence of Lemma \ref{meas} and the equation \eqref{GKE1} for $\omega_{\tilde\Sigma}$ is the following
\begin{lem}\label{meas1}
There exists a positive constant $C_2$ such that, for any $K\subset\tilde\Sigma_{reg}$,
$$\nu(K)=C_2\int_{K}\omega_{\tilde\Sigma}^2.$$
\end{lem}
Note that it follows from the above proof that $\nu(Z)=\underline{V}_\infty(z,D)=C_1\int_{\tilde X}e^{\tilde\psi}\tilde\Omega>0$, where $D$ is a diameter bound of $(Z,d_Z)$. Then we can apply the same arguments in \cite[Theorem 1.2]{GTZ13} to conclude
\begin{lem}\label{meas2}
$\nu(Z\setminus Z_0)=0$.
\end{lem}
Now using Lemmas \ref{meas1}, \ref{meas2} and the asymptotics of $\omega_{\tilde\Sigma}$ (see subsection \ref{subs3.1'}), the same arguments in \cite[Section 2]{ToZy17} (also see \cite[Section 3]{GTZ}) give
\begin{lem}\label{meas3}
\begin{itemize}
\item[(1)]The Hausdorff dimension $dim_{\mathcal{H}}(Y\setminus Z_0)\le2$;
\item[(2)]$\nu=C_2\mathcal{H}_{d_Z}^{4}$ on $Z$ as measures;
\item[(3)]$\nu_{-1}(Z\setminus Z_0)=0$, where $\nu_{-1}$, the measure in codimension one, is defined by, for any subset $K\subset Z$,
$$\nu_{-1}(K)=\lim_{\delta\to0}\inf\left\{\sum_{i}r_i^{-1}\nu(B_{d_Z}(z_i,r_i))|K\subset\cup_{i}B_{d_Z}(z_i,r_i), r_i<\delta\right\}.$$
\end{itemize}
\end{lem}
Finally, we are able to complete the proof of Theorem \ref{result}.
\begin{proof}[Proof of Theorem \ref{result}]
Firstly, the arguments in Lemma \ref{lemonbase} imply that, for any $\tilde s_1,\tilde s_2\in\tilde\Sigma_{reg}$ and $\epsilon>0$, there exists a piecewise smooth curve $\tilde\gamma_\epsilon\subset\tilde\Sigma_{reg}$ connecting $\tilde s_1,\tilde s_2$ with
$$L_{\tilde d}(\tilde\gamma_\epsilon)=L_{\omega_{\tilde\Sigma}}(\tilde\gamma_\epsilon)\le \tilde d(\tilde s_1,\tilde s_2)+\epsilon.$$
Therefore,
\begin{align}
d_Z(h(\tilde s_1),h(\tilde s_2))&\le L_{d_Z}(h(\tilde\gamma_\epsilon))\nonumber\\
&=L_{\omega_{\tilde\Sigma}}(\tilde\gamma_\epsilon)\nonumber\\
&\le \tilde d_\infty(\tilde s_1,\tilde s_2)+\epsilon\nonumber.
\end{align}
Hence, since $\epsilon$ is arbitrary, we have
\begin{equation}\label{isom1}
d_Z(h(\tilde s_1),h(\tilde s_2))\le\tilde d_\infty(\tilde s_1,\tilde s_2).
\end{equation}
On the other hand, for any $z_1=h(\tilde s_1),z_2=h(\tilde s_2)\in Z_0$ and $\epsilon>0$, we choose $\epsilon$ small enough such that $B_{d_Z}(z_2,\epsilon)\subset Z_0$ and is geodesiccally convex. By using item (3) in Lemma \ref{meas3}, we conclude from \cite[Theorem 3.7]{CC2} that there exists a $z_3\in B_{d_Z}(z_2,\epsilon)$ and a minimal geodesic $\tau_\epsilon\subset Z_0$ connecting $z_1$ and $z_3$. We can also connecting $z_2$ and $z_3$ by a minimal geodesic $\tau_\epsilon'\subset B_{d_Z}(z_2,\epsilon)\subset Z_0$. Hence, by connecting $\tau_\epsilon$ and $\tau_\epsilon'$ we obtain a curve $\tau_\epsilon''$ connecting $z_1,z_2$ with
\begin{align}
L_{d_Z}(\tau_\epsilon'')&\le L_{d_Z}(\tau_\epsilon)+\epsilon\nonumber\\
&\le d_Z(z_1,z_3)+\epsilon\nonumber\\
&\le d_Z(z_1,z_2)+2\epsilon\nonumber.
\end{align}
Therefore,
\begin{align}
\tilde d_\infty(\tilde s_1,\tilde s_2)&\le L_{\omega_{\tilde\Sigma}}(h^{-1}(\tau_\epsilon''))\nonumber\\
&=L_{d_Z}(\tau_\epsilon'')\nonumber\\
&\le d_Z(z_1,z_2)+2\epsilon\nonumber\\
&= d_Z(h(\tilde s_1),h(\tilde s_2))+2\epsilon\nonumber.
\end{align}
Letting $\epsilon\to0$ gives
\begin{equation}\label{isom2}
\tilde d_\infty(\tilde s_1,\tilde s_2)\le d_Z(h(\tilde s_1),h(\tilde s_2)).
\end{equation}
From \eqref{isom1} and \eqref{isom2} we see that $h:(\tilde\Sigma_{reg},\tilde d)\to (Z_0,d_Z)$ is an isometry. But $Z_0$ is dense in $Z$ and $\tilde\Sigma_{reg}$ is dense in $\tilde \Sigma$, we can extend $h$ to an isometry $\tilde h:(\tilde\Sigma,\tilde d_\infty)\to(Z,d_Z)$. Thus, we have proved that every Gromov-Hausdorff limit of $(\tilde X,\omega(t))$ is isometric to $(\tilde\Sigma,\tilde d_\infty)$ and hence $(\tilde X,\omega(t))\to(\tilde\Sigma,\tilde d_\infty)$ in Gromov-Hausdorff topology without passing to a subsequence.
\par Theorem \ref{result} is proved.
\end{proof}

We end this paper by a remark.
\begin{rem}\label{ending}
Theorem \ref{result} (assuming Ricci curvature lower bound) and  Theorem \ref{result.1} hold in some similar settings. Let $f':X'\to B$ be a holomorphic submersion with connected fibers between two compact connected K\"ahler manifolds with $dim(Y)> dim(B)\ge1$ and $f'^*\chi_B\in2\pi c_1(K_{X'})$ for some K\"ahler metric $\chi_{B}$ on $B$. Set
$$\tilde f:=(f,f'):X\times X'\to \Sigma\times B.$$
Then Theorems \ref{result} holds in this setting. Moreover, given a holomorphic surjective map $\tilde f:X_1\times X_2\times\ldots\times X_q\to B_1\times B_2\times\ldots\times B_q$, which is a product of $q$ holomorphic surjective maps $f_i:X_i\to B_i$ with connected fibers and $f_i^*\chi_i\in2\pi c_1(K_{X_i})$ for some K\"ahler metric $\chi_i$ on $B_i$, assume all the involved spaces are compact connected K\"ahler manifolds and $f_i:X_i\to B_i$ is either a holomorphic submersion or $dim(B_i)=1$. Then Theorem \ref{result} and \ref{result.1} hold on $\tilde f:X_1\times X_2\times\ldots\times X_q\to B_1\times B_2\times\ldots\times B_q$.
\end{rem}

\subsection*{Acknowledgements}
The author is grateful to Professor Huai-Dong Cao for constant encouragement and support and Professor Valentino Tosatti for his crucial help during this work, valuable suggestions on a previous draft and constant encouragement. He also thanks Professor Hans-Joachim Hein for communications which motivate Remark \ref{remk1.2}, Professor Chengjie Yu for useful comments on a previous draft, Professor Zhenlei Zhang for collaboration and encouragement and Peng Zhou for kind help. This work was carried out while the author was visiting Department of Mathematics at Northwestern University, which he would like to thank for the hospitality.
\par The author is grateful to the referee and editor for their careful reading and very useful suggestions and corrections, which help to improve this paper.
\par  Very recently, Fu, Guo and Song \cite{FGS} made a big progress on studying the geometry of the continuity method. They proved that the diameter of $\omega(t)$ solving from the continuity method \eqref{LT0} or \eqref{LT0.1} is uniformly bounded. Their result in particular gives an alternative proof for the diameter upper bound of the continuity method involved in proofs of Theorems \ref{result0.1}, \ref{result.1} and \ref{result0.2}.

\end{document}